\documentclass[12pt]{amsart} 

\usepackage{amsmath} \usepackage{amssymb} \usepackage{amsthm}
\usepackage{amscd} 
\usepackage[all]{xy} 
\usepackage{enumerate}
\usepackage{mathrsfs}
\usepackage{multirow}
\usepackage{bigdelim}
\usepackage{color}
\usepackage[bookmarks=false]{hyperref}

\def\o{{\omega}}
\def\O{{\mathcal{O}}}
\def\ol{\overline}

\theoremstyle{plain}
\newtheorem{thm}{Theorem}[section] 

\newtheorem{prop}[thm]{Proposition}
\newtheorem{lem}[thm]{Lemma}
\theoremstyle{definition} 
\newtheorem{defn}[thm]{Definition}
\newtheorem{notation}[thm]{Notation}
\theoremstyle{remark}
\newtheorem{eg}[thm]{Example} 
\newtheorem{rem}[thm]{Remark}
\newtheorem{obs}[thm]{Observation}
\newtheorem{cl}[thm]{Claim}
\newtheorem*{acknowledgement}{Acknowledgments}

\title{When is the Albanese morphism an algebraic \\fiber space in positive characteristic?}
\author{Sho Ejiri}
\address{Department of Mathematics, Graduate School of Science, Osaka University, Toyonaka, Osaka 560-0043, Japan}
\email{s-ejiri@cr.math.sci.osaka-u.ac.jp}

\begin{document}
\tolerance = 9999
\maketitle
\markboth{SHO EJIRI}{} 

\begin{abstract}
In this paper, we study the Albanese morphisms in positive characteristic.
We prove that the Albanese morphism of a variety with nef anti-canonical divisor is an algebraic fiber space, 
under the assumption that the general fiber is $F$-pure. 
Furthermore, we consider a notion of $F$-splitting for morphisms, and investigate it in the case of Albanese morphisms.
We show that an $F$-split variety has $F$-split Albanese morphism, and that the $F$-split Albanese morphism is an algebraic fiber space. 
As an application, we provide a new characterization of abelian varieties.
\end{abstract}
\section{Introduction} \label{section:intro} %%%%%%%%%%%%%%%%%%%%%%%%%%%%%%%%%%%%%%%%%%%%%%%%%%%%%%
The Albanese morphism is an important tool 
in the study of a variety with non-positive Kodaira dimension. 
In characteristic zero, Kawamata proved that 
the Albanese morphism of a smooth projective variety with Kodaira dimension zero 
is an algebraic fiber space~\cite[Theorem~1]{Kaw81}. 
Zhang showed that the same holds in the case when the anti-canonical divisor is nef~\cite[Corollary~2]{Zha05}. 
In the same case, Lu, Tu, Zhang and Zheng proved that 
the Albanese morphism is a flat morphism with reduced fibers~\cite{LTZZ}, 
and Cao showed recently that it is actually locally isotrivial~\cite{Cao16}. 
In positive characteristic, Hacon and Patakfalvi proved that the Albanese morphism of a smooth projective variety $X$ 
is surjective if the $S$-Kodaira dimension $\kappa_S(X)$ of $X$ is zero~\cite[Theorem 1.1.1]{HP16a} (cf.~\cite{HPZ19}).
Here, the $S$-Kodaira dimension is a positive characteristic analogue of the usual Kodaira dimension. 
Wang showed that the Albanese morphism of a threefold with semi-ample anti-canonical divisor 
is surjective if the general fiber is $F$-pure~\cite{Wan16}. 

In this paper, we generalize his result to varieties of arbitrary dimension, 
which can be viewed as a positive characteristic counterpart of the above result of Zhang. 
\begin{thm}\label{thm:nef_intro}
Let $X$ be a normal projective variety over an algebraically closed field of characteristic $p>0$, 
and $\Delta$ an effective $\mathbb{Q}$-Weil divisor on $X$ such that 
$-m(K_X+\Delta)$ is a nef Cartier divisor for an integer $m>0$ not divisible by $p$. 
Let $a:X\to A$ denote the Albanese morphism of $X$, and $X_{\ol\eta}$ the geometric generic fiber over the image of $a$.
If $(X_{\ol\eta},\Delta|_{X_{\ol\eta}})$ is $F$-pure, then $a$ is an algebraic fiber space.
\end{thm}
We also study the relationship between the Albanese morphism of $X$ and Frobenius splittings on $X$. 
The notion of a Frobenius split variety was introduced by Mehta and Ramanathan~\cite{MR85}, 
which is considered to be closely related to a variety of Calabi--Yau type~\cite{GOST,GT16,Oka17,SS10}.  
We consider a generalization of this notion to pairs $(f,\Gamma)$ 
consisting of a morphism $f:V\to W$ between varieties 
and an effective $\mathbb{Q}$-Weil divisor $\Gamma$ on $V$ (Definition~\ref{defn:relFsp}). 
In this paper, we focus on a Frobenius splitting of the Albanese morphism.
Let $X$ be a normal projective variety over an algebraically closed field $k$ of characteristic $p>0$, 
let $\Delta$ be an effective $\mathbb{Q}$-Weil divisor on $X$, 
and let $a:X\to A$ be the Albanese morphism of $X$. 
Then there is the following relationship between Frobenius splittings of $(X,\Delta)$ and that of $(a,\Delta)$:
\begin{thm}\label{thm:ch_Fsp_intro}
The pair $(X,\Delta)$ is $F$-split if and only if $(a,\Delta)$ is $F$-split and $A$ is ordinary.
\end{thm}
We study the Albanese morphism $a$ under the assumption that $(a,\Delta)$ is locally $F$-split (Definition~\ref{defn:relFsp}), 
which is weaker than the assumption that the pair is $F$-split.
For instance, a flat morphism with normal $F$-split fibers is locally $F$-split, but not necessarily $F$-split. 
The next theorem shows that, if $(a,\Delta)$ is locally $F$-split, 
then $a$ is an algebraic fiber space whose fibers satisfy certain geometric properties.
\begin{thm}\label{thm:albf_intro}
Assume that $(a,\Delta)$ is locally $F$-split  
and that $m\Delta$ is Cartier for an integer $m>0$ not divisible by $p$. 
Then the following hold:
\begin{itemize}
\item[$(1)$] The morphism $a$ is an algebraic fiber space. 
\item[$(2)$] The support of $\Delta$ does not contain any irreducible component of any fiber. 
\item[$(3)$] For every scheme-theoretic point $z\in A$, the pair $(X_{\ol z},\Delta_{\ol z})$ is $F$-split, 
where $X_{\ol z}$ is the geometric fiber over $z$. In particular, $X_{\ol z}$ is reduced.
\item[$(4)$] The morphism $a$ is smooth in codimension one. 
In other words, there exists an open subset $U$ of $X$ such that $\mathrm{codim}(X\setminus U)\ge 2$ and $a|_U:U\to A$ is a smooth morphism. 
In particular, the general fiber of $a$ is normal.  
\end{itemize}
\end{thm}
One can recover the result of Hacon and Patakfalvi when $K_X$ is numerically trivial, 
because in this case, the condition $\kappa_S(X)=0$ is equivalent to saying that $X$ is $F$-split. 
As a corollary of Theorem~\ref{thm:albf_intro}, we provide a new characterization of abelian varieties. 
\begin{thm}\label{thm:ch_ab_intro}
Assume that $(a,\Delta)$ is locally $F$-split $($resp. $(X,\Delta)$ is $F$-split$)$. Then $\dim A\le\dim X$. 
Furthermore, the equality holds if and only if $X$ is an abelian variety $($resp. ordinary abelian variety$)$ and $\Delta=0$.
\end{thm}
Applying this theorem, we also give a necessary and sufficient condition 
for a normal projective variety to have $F$-split Albanese morphism (Theorem~\ref{thm:nscond}).
We conclude this paper with a classification of minimal surfaces 
whose Albanese morphisms are $F$-split or locally $F$-split (Theorem~\ref{thm:surf}). 
\begin{small}
\begin{acknowledgement}
The author wishes to express his gratitude to his supervisor Professor Shunsuke Takagi for suggesting problems, 
valuable comments and helpful advice. 
He is deeply grateful to Professors Zsolt Patakfalvi and Yoshinori Gongyo for fruitful discussions and valuable comments.
He would like to thank Professors Osamu Fujino, Nobuo Hara and Doctor Yuan Wang for stimulating discussions, questions and comments.
He also would like to thank the reviewer for a careful reading and helpful suggestions. 
Part of this work was carried out during his visit to Princeton University with support from 
The University of Tokyo/Princeton University Strategic Partnership Teaching and Research Collaboration Grant, 
and from the Program for Leading Graduate Schools, MEXT, Japan. 
He was also supported by JSPS KAKENHI Grant Number 15J09117.
He would like to thank Professor Adrian Langer for pointing out an error in Proposition~\ref{prop:fgh}. 
\end{acknowledgement}
\end{small}
\section{Notation and conventions} \label{section:notat} %%%%%%%%%%%%%%%%%%%%%%%%%%%%%%%%%%%%%%%%%%%%%%%%%%%%
Throughout this paper, we fix an algebraically closed field $k$ of characteristic $p>0$. 
By \textit{$k$-scheme}, we mean a separated scheme of finite type over $\mathrm{Spec}\,k$. 
An integral $k$-scheme is called a \textit{variety over $k$}. 
Let $X$ be a normal variety. 
A Weil divisor $D$ is said to be $\mathbb{Q}$-Cartier (resp. $\mathbb{Z}_{(p)}$-Cartier) if 
$mD$ is Cartier for some $0<m\in\mathbb{Z}$ (resp. $0<m\in \mathbb{Z}\setminus p\mathbb{Z}$). 
Here, $\mathbb{Z}_{(p)}$ is the localization of $\mathbb{Z}$ at $(p)=p\mathbb{Z}$. 
Let $\varphi:S\to T$ be a morphism of schemes and $T'$ a $T$-scheme. 
Then, $S_{T'}$ and $\varphi_{T'}:S_{T'}\to T'$ denote respectively 
the fiber product $S\times_{T}T'$ and its second projection. 
For a Cartier, $\mathbb{Z}_{(p)}$-Cartier or $\mathbb{Q}$-Cartier divisor $D$ on $S$ (resp. an $\O_S$-module $\mathcal G$), 
the pullback of $D$ (resp. $\mathcal G$) to $S_{T'}$
is written as $D_{T'}$ (resp. $\mathcal G_{T'}$),  
if it is well-defined. 
Similarly, for a homomorphism $\alpha:\mathcal F\to\mathcal G$ of $\O_S$-modules, 
$\alpha_{T'}:\mathcal F_{T'}\to \mathcal G_{T'}$ is the pullback of $\alpha$ to $S_{T'}$. 
Let $f:X\to Z$ be a morphism between $k$-schemes. 
Then, $F_X$ denotes the absolute Frobenius morphism of $X$. 
The source of $F_X^e$ is often written as $X^e$. 
The morphism $f:X\to Z$ is denoted by $f^{(e)}:X^e\to Z^e$ 
when we regard $X$ and $Z$ as $X^e$ as $Z^e$, respectively. 
We write the induced morphism $(F_X^e, f^{(e)}):X^e\to X\times_{Z} Z^e=:X_{Z^e}$ as $F^{(e)}_{X/Z}$.
\section{Trace maps of relative Frobenius morphisms} \label{section:prelim} %%%%%%%%%%%%%%%%%%%%%%%%%%%%%%%%%%%%%%%%%%%%%%%%%
In this section, 
given a morphism between varieties over an algebraically closed field $k$, 
we consider the relative Frobenius morphism and its trace map. 
\subsection{Base change by Frobenius morphisms}\label{subsection:bc}
Let $f:X\to Z$ be a morphism (not necessarily surjective) between $k$-schemes. 
For each integer $e>0$, we have the following diagram:
$$ 
\xymatrix@R=25pt@C=25pt{ 
X^e \ar[dr]^{F_{X}^{e}} \ar[d]_{F_{X/Z}^{(e)}} & \\ 
X_{Z^e} \ar[r]_{(F_Z^e)_X} \ar[d]_{f_{Z^e}} & X \ar[d]^{f} \\ 
Z^e \ar[r]^{F_Z^e} & Z}
$$
We first consider the properties of $X_{Z^e}$ when $Z$ is a smooth variety. 
\begin{lem}\label{lem:bc}
With the notation above, assume that $Z$ is a smooth variety. 
\begin{itemize}
\item[$($1$)$] If $X$ is a Gorenstein $k$-scheme of pure dimension, then so is $X_{Z^e}$. 
Furthermore, the dualizing sheaf $\o_{X_{Z^e}}$ of $X_{Z^e}$ is isomorphic to 
${f_{Z^e}}^*\o_{Z^e}^{1-p^e}\otimes(\o_X)_{Z^e}$, 
where $\o_X$ is the dualizing sheaf of $X$. 
\item[$($2$)$] Suppose that $f$ is dominant and separable. 
If $X$ is a variety, then so is $X_{Z^e}$.
\end{itemize}
\end{lem}
\begin{proof}
We first note that since $(F_Z^e)_X$ is homeomorphic, 
if $X$ is of pure dimension (resp. irreducible), then so is $X_{Z^e}$. 
Now, since $F_Z^e$ is a Gorenstein morphism~\cite[V,\S 9]{RD}, 
the base change $(F_Z^e)_X$ has the same property, 
so $X_{Z^e}$ is a Gorenstein $k$-scheme. 
Furthermore, by~\cite[Theorems 3.6.1 and 4.3.3]{Con00}, we have 
$$
\o_{X_{Z^e}}
\cong\o_{(F_{Z}^e)_X}\otimes(\o_X)_{Z^e}
\cong{f_{Z^e}}^*\o_{F_Z^e}\otimes(\o_X)_{Z^e}
\cong{f_{Z^e}}^*\o_{Z^e}^{1-p^e}\otimes(\o_X)_{Z^e},
$$ 
which proves~(1). 
Next, we show~(2). We may assume that $X=\mathrm{Spec}\,A$ and $Z=\mathrm{Spec}\,B$. 
Let $K$ be the function field of $X$. 
The separability of $f$ implies that $F^e_*B\otimes_{B}K$ is reduced. 
Since $F_Z^e$ is flat, $F^e_*B\otimes_{B}A\to F^e_*B\otimes_{B}K$ is injective, so $F^e_*B\otimes_{B}A$ is reduced. 
\end{proof}
\subsection{Trace maps}\label{subsection:trace} \quad
 
\noindent(\ref{subsection:trace}.1) 
Let $\pi:X\to Z$ be a finite surjective morphism between Gorenstein $k$-schemes of pure dimension, 
and let $\o_X$ and $\o_Z$ be dualizing sheaves of $X$ and $Z$, respectively. 
We denote by $\mathrm{Tr}_{\pi}:\pi_*\o_X\to\o_Z$ the morphism obtained by applying the functor $\mathcal Hom_Z(\underline{\quad},\o_Z)$ 
to the natural morphism $\pi^{\#}:\O_Z\to\pi_*\O_X$. 
This is called \textit{the trace map} of $\pi$. 
 
\noindent(\ref{subsection:trace}.2) 
Let $f:X\to Z$ be a morphism (not necessarily surjective) from a Gorenstein variety $X$ to a smooth variety $Z$. 
Then it follows from Lemma \ref{lem:bc} that, 
\begin{align*}
\o_{X^e}\otimes{F_{X/Z}^{(e)}}^*\o_{X_{Z^e}}^{-1}
\cong \o_{X^e}\otimes{f^{(e)}}^*\o_{Z^e}^{p^e-1}\otimes\o_{X^e}^{-p^e}
\cong\o_{X^e/Z^e}^{1-p^e}. \label{isom:a}\tag{1}
\end{align*}
Hence, by the projection formula, we get 
$({F_{X/Z}^{(e)}}_*\o_{X^e})\otimes\o_{X_{Z^e}}^{-1}\cong {F_{X/Z}^{(e)}}_*\o_{X^e/Z^e}^{1-p^e}.$
Define $$\phi^{(e)}_{X/Z}:=\mathrm{Tr}_{F_{X/Z}^{(e)}}\otimes\o_{X_{Z^e}}^{-1}:{F_{X/Z}^{(e)}}_*\o_{X^e/Z^e}^{1-p^e}\to\O_{X_{Z^e}}.$$
Note that we now have the following isomorphisms:
\begin{align*}
{F_{X/Z}^{(e)}}_*\o_{X^e/Z^e}^{1-p^e}
&\cong{F_{X/Z}^{(e)}}_*\mathcal Hom\left({F_{X/Z}^{(e)}}^*\o_{X_{Z^e}},\o_{X^e}\right) \hspace{50pt}\textup{{\tiny by (\ref{isom:a})}}\\
&\cong\mathcal Hom\left(\left({F_{X/Z}^{(e)}}_*\O_{X^e}\right)\otimes\o_{X_{Z^e}},\o_{X_{Z^e}}\right) \hspace{25pt}\textup{{\tiny by the Grothendieck duality}}\\
&\cong\mathcal Hom\left({F_{X/Z}^{(e)}}_*\O_{X^e},\O_{X_{Z^e}}\right). 
\end{align*}
\noindent(\ref{subsection:trace}.3) 
Let $f:X\to Z$ be a morphism from a normal variety $X$ to a smooth variety $Z$. 
Let $\iota:U\to X$ be the open immersion from the regular locus $U$ of $X$. 
We then have 
$$
{\iota_{Z^e}}_*{F_{U/Z}^{(e)}}_*\o_{U^e/Z^e}^{1-p^e}
\cong {F_{X/Z}^{(e)}}_*\o_{X^e/Z^e}^{1-p^e},
\quad\textup{and}\quad
{\iota_{Z^e}}_*\O_{U_{Z^e}} \cong\O_{X_{Z^e}}.
$$
Hence, we can define 
$$
\phi^{(e)}_{X/Z}:=\iota_*\phi^{(e)}_{U/Z}:{F_{X/Z}^{(e)}}_*\o_{X^e/Z^e}^{1-p^e}\to\O_{X_{Z^e}}.
$$
Let $K_{X/Z}$ be a Weil divisor on $X$ such that $\O_X(K_{X/Z})\cong\o_{X/Z}$. 
Let $\Delta$ be an effective $\mathbb{Q}$-Weil divisor on $X$. 
For every $e>0$, we define 
\begin{align*}
\mathcal L^{(e)}_{(X/Z,\Delta)}:
&=\O_{X^e}( \lfloor (1-p^e)(K_{X^e/Z^e}+\Delta) \rfloor )
\subseteq \O_{X^e}( (1-p^e)K_{X^e/Z^e} ),\quad \textup{and} \\
\phi^{(e)}_{(X/Z,\Delta)}:
&{F_{X/Z}^{(e)}}_*\mathcal L_{(X/Z,\Delta)}^{(e)}
\hookrightarrow {F_{X/Z}^{(e)}}_* \O_{X^e}((1-p^e)K_{X^e/Z^e})
\xrightarrow{ \phi^{(e)}_{X/Z} } \O_{X_{Z^e}}.
\end{align*}
It is easily seen that the above morphism is also obtained by applying the functor 
$\mathcal Hom_{X_{Z^e}}(\underline{\quad},\O_{X_{Z^e}})$ 
to the natural morphism 
$$\O_{X_{Z^e}}\to {F_{X/Z}^{(e)}}_*\O_{X^e}\hookrightarrow {F_{X/Z}^{(e)}}_*\O_{X^e}(\lceil(p^e-1)\Delta\rceil).$$
When $Z=\mathrm{Spec}\,k$, we may identify $Z^e$, $X_{Z^e}$ and $F_{X/Z}^{(e)}$ with $Z$, $X$ and $F_X^e$, respectively. 
In this case, we denote $\phi_{(X/Z,\Delta)}^{(e)}$ by $\phi_{(X,\Delta)}^{(e)}$.
\section{Varieties with nef anti-canonical divisor}\label{section:nef} %%%%%%%%%%%%%%%%%%%%%%%%%%%%%%%%%%%%%%%%%%
In this section, we prove Theorem~\ref{thm:nef_intro} which states that 
the Albanese morphism of a normal projective variety 
with nef anti-canonical divisor is an algebraic fiber space 
if the geometric generic fiber is $F$-pure. 
Throughout this section, we work over an algebraically closed field $k$ of characteristic $p>0$. 
To begin with, we recall the \textit{Albanese morphism} of a normal projective variety $X$. 
This is a morphism $a:X\to A$ to an abelian variety $A$ (called the \textit{Albanese variety}) such that 
every morphism $b:X\to B$ to another abelian variety $B$ factors uniquely through $a$.  
The existence of the Albanese morphism for a normal projective variety is proved for instance in~\cite[\S 9]{FGA}.  
We must remark that \textbf{the above morphism is different in general 
from the Albanese map that is defined by using the Albanese morphism 
of a resolution of singularities}.    

We next recall the definition of weak positivity. 
Note that our definition is slightly different from the original one introduced by Viehweg~\cite{Vie83}.
\begin{defn}
Let $\mathcal G$ be a coherent sheaf on a normal quasi-projective variety $Y$, 
and let $\eta$ denote the generic point of $Y$. 
We say that $\mathcal G$ is \textit{weakly positive} if 
for every ample line bundle $\mathcal H$ on $Y$ and each positive number $\alpha$, 
there exists a positive number $\beta$  
such that the natural map 
$$ 
H^0\left( Y, \left( S^{\alpha\beta}\mathcal G \right)^{**} \otimes \mathcal H^{\beta} \right) 
\otimes \O_{Y,\eta}
\to \left( \left( S^{\alpha\beta}\mathcal G \right)^{**} \otimes \mathcal H^{\beta} \right)_{\eta}
$$ 
is surjective. Here, $S^{\alpha\beta}(\underline{\quad})$ and $(\underline{\quad})^{**}$ denote 
the $\alpha\beta$-th symmetric product and the double dual, respectively. 
\end{defn}
It follows from the definition that a line bundle on a projective variety is weakly positive if and only if it is pseudo-effective. 
 
Theorem~\ref{thm:nef_intro} is an application of Theorems~\ref{thm:wp} and~\ref{thm:HP} below. 
\begin{thm}[\textup{\cite{Eji19p}}]\label{thm:wp}
Let $f:X\to Y$ be a surjective morphism between normal varieties.
Let $\Delta$ be an effective divisor on $X$ such that $m(K_X+\Delta)$ is 
an integral Cartier divisor for an integer $m>0$ not divisible by $p$.
Assume that $(X_{\ol\eta},\Delta|_{X_{\ol\eta}})$ is $F$-pure, where $X_{\ol\eta}$ is the geometric generic fiber.
If $-(K_X+\Delta+f^*D)$ is a nef $\mathbb{Q}$-Cartier divisor on $X$ for some $\mathbb{Q}$-Cartier divisor $D$ on $Y$, 
then $\O_Y(-n(K_Y+D))$ is weakly positive for an integer $n>0$ such that $nD$ is integral.
\end{thm}
\begin{thm}[\textup{\cite{HP16b}}]\label{thm:HP}
Let $X$ be a normal projective variety with $\kappa(X,K_X)=0$. 
Let $a:X\to A$ be the Albanese morphism of $X$. 
If $a:X\to \mathrm{Im}(a)$ is generically finite and separable, then $a$ is surjective.
\end{thm}
The next lemma is also used in the proof of Theorem~\ref{thm:nef_intro}.
\begin{lem}\label{lem:wp}
Let $D$ be an effective Weil divisor on a normal projective variety $Y$.
If $\O_Y(-D)$ is weakly positive, then $D=0$.
\end{lem}
\begin{proof}
Thanks to \cite[4.1]{deJ96}, we can find a regular alteration $\pi:Y' \to Y$ of $Y$, 
that is, a generically finite morphism $\pi$ from a smooth projective variety $Y'$ to $Y$. 
Let $V$ denote the regular locus of $Y$ and set $V':=\pi^{-1}(V)$. 
Let $D'\ge0$ be a divisor on $Y'$ such that $D'|_{V'}=(\pi|_{V'})^*(D|_V)$. 
Take an ample Cartier divisor $H$ on $Y$ and $\alpha \in \mathbb Z_{>0}$. 
By the weak positivity of $\O_Y(-D)$, we see that
$h^0(Y, \O_Y(-\alpha \beta D +\beta H)) \ne 0$ for some $\beta \in \mathbb Z_{>0}$, 
so there is a rational function $\varphi \in K(Y)$ 
with $\beta( -\alpha D +H ) +(\varphi) \ge 0$. 
We then have a divisor $E \ge0$ on $Y'$ 
such that $\mathrm{Supp}\,E \subseteq Y' \setminus V'$
and $\beta( -\alpha D' +\pi^*H) +(\pi^*\varphi) +E \ge0$. 
Since $\dim \left( \pi(\mathrm{Supp}\,E) \right) \le d-2$, 
where $d:=\dim Y$, 
we get $E \cdot (\pi^* H)^{d-1}=0$, 
and so $(-\alpha D' +\pi^* H)\cdot (\pi^*H)^{d-1} \ge 0.$ 
This means $(-D') \cdot (\pi^*H)^{d-1} \ge 0,$ 
which can occur only when $D'=0$, 
and hence $D=0$. 
\end{proof}
\begin{proof}[Proof of Theorem~\ref{thm:nef_intro}]
Let $Z$ be the normalization of $\mathrm{Im}(a)$ and $f:X\to Z$ be the induced morphism. 
We now have the natural morphism $\Omega_A^1|_Z\to \Omega_Z^1$ that is generically surjective, 
so $H^0(Z,\o_Z)\ne0$. 
We see from Theorem~\ref{thm:wp} that $\o_Z^{-1}$ is weakly positive, 
so we get $\o_Z\cong\O_Z$ by Lemma~\ref{lem:wp}, 
and hence $a$ is surjective (i.e., $Z=A$) by Theorem~\ref{thm:HP}. 
Let $a:X\xrightarrow{g}Y\xrightarrow{h}A$ be the Stein factorization of $a$. 
Since the geometric generic fiber of $a$ is $F$-pure, it is reduced, so $a$ is separable, 
which implies $h$ is a separable finite morphism. 
Therefore, we get $\o_Y\cong \o_{Y/A}\cong\O_Y(R)$, where $R$ is the ramification divisor of $h$.  
We then find that $R=0$ by the same argument as before, 
and so $h$ is an \'etale morphism by the Zariski--Nagata purity theorem. 
Hence, we see that $Y$ is an abelian variety by~\cite[Section 18, Theorem]{Mum70} and that $h$ is an isomorphism, 
which is our assertion. 
\end{proof}
\section{Frobenius split morphisms}\label{section:relFsp} %%%%%%%%%%%%%%%%%%%%%%%%%%%%%%%%%%%%%%%
In this section, we introduce and study the notion of an $F$-split morphism.
We fix an algebraically closed field $k$ of characteristic $p>0$. 
\begin{defn}\label{defn:relFsp}
Let $X$ be a normal variety and $\Delta$ an effective $\mathbb{Q}$-Weil divisor on $X$. 
Let $f:X\to Z$ be a projective morphism to a smooth variety $Z$. 
We say that $f$ is \textit{sharply $F$-split} (\textit{$F$-split} for short) 
with respect to $\Delta$ if there exists an $e>0$ such that the composite 
\begin{align}
\O_{X_{Z^e}}\xrightarrow{{F_{X/Z}^{(e)}}^\sharp} {F_{X/Z}^{(e)}}_*\O_{X^e}\hookrightarrow{F_{X/Z}^{(e)}}_*\O_{X^e}(\lceil(p^e-1)\Delta\rceil) \tag*{$(\ref{defn:relFsp}.1)_e$}
\end{align}
of the natural homomorphism ${F_{X/Z}^{(e)}}^\sharp$ and 
the natural inclusion ${F_{X/Z}^{(e)}}_*\O_{X^e}\hookrightarrow{F_{X/Z}^{(e)}}_*\O_{X^e}(\lceil(p^e-1)\Delta\rceil)$ 
is injective and splits as an $\O_{X_{Z^e}}$-module homomorphism. 
We say that $f$ is \textit{locally sharply $F$-split} (\textit{locally $F$-split} for short) 
with respect to $\Delta$ if there exists an open covering $\{V_i\}$ of $Z$ such that 
$f|_{f^{-1}(V_i)}:f^{-1}(V_i)\to V_i$ is $F$-split with respect to $\Delta|_{f^{-1}(V_i)}$ for every $i$.

We often say that the pair $(f,\Delta)$ is $F$-split (resp. locally $F$-split) 
if $f$ is $F$-split (resp. locally $F$-split) with respect to $\Delta$. 
We simply say $f$ is $F$-split (resp. locally $F$-split) if so is $(f,0)$.
\end{defn}
\begin{rem}\label{rem:relFsp} 
\quad\\
(1) If the morphism $(\ref{defn:relFsp}.1)_e$ splits, then $(\ref{defn:relFsp}.1)_{en}$ also splits for every integer $n>0$. \\
(2) When $Z=\mathrm{Spec}\,k$, where we recall that $k$ is assumed to be algebraically closed, 
it is easily seen that $(f,\Delta)$ is $F$-split if and only if $(X,\Delta)$ is $F$-split. \\
(3) Let $\Delta'$ be an effective $\mathbb{Q}$-divisor on $X$ with $\Delta'\le\Delta$. 
If $(f,\Delta)$ is $F$-split (resp. locally $F$-split), then so is $(f,\Delta')$. \\
(4) In~\cite{Has10}, Hashimoto has dealt with morphisms with local splittings of $(\ref{defn:relFsp}.1)_e$ 
\end{rem}
As an example, we consider the case of projective bundles. 
\begin{eg}\label{eg:pn-bdl}
Let $X$ be the projective bundle associated with a vector bundle $\mathcal E$ on a normal projective variety $Z$, 
and let $f:X\to Z$ be the projection. 
Then $f$ is locally $F$-split. 
Furthermore, if $\mathcal E\cong\bigoplus_{i=1}^{n}\mathcal L_i$, where $\mathcal L_1,\ldots,\mathcal L_n$ are line bundles on $Z$, then $f$ is $F$-split.
The first statement follows from the second. 
For every $m\ge0$, there exists the natural injective morphism 
$$\psi_{m}:\bigoplus_{m_1+\cdots+m_n=m}\mathcal L_i^{m_ip}\cong {F_Z}^*S^m\mathcal E\to S^{mp}\mathcal E.$$
Then, the image of $\psi_m$ is obviously 
$\bigoplus_{m_1+\cdots+m_n=m}\mathcal L_i^{m_ip}\subseteq S^{mp}\mathcal E$, and hence $\psi_m$ splits.
The morphism $\O_{X_{Z^1}}\to{F_{X/Z}^{(1)}}_*\O_{X^1}$ corresponds to the morphism
$$\psi:=\bigoplus_{m\ge0}\psi_m :\bigoplus_{m\ge0} S^m{F_Z^1}^*\mathcal E \to \bigoplus_{m\ge0} S^{mp}\mathcal E\subseteq\bigoplus_{m\ge0} S^{m}\mathcal E.$$
Since $\psi_m$ splits for every $m\ge0$, we see that $\psi$ also splits, 
and hence so does $\O_{X_{Z^1}}\to {F_{X/Z}^{(1)}}_*\O_{X^1}$. 
Note that, as we see in Theorem~\ref{thm:surf}, 
there exists a projective bundle over an elliptic curve 
whose projection is not an $F$-split morphism. 
\end{eg}
Next, we prove that if $f:X\to Z$ is $F$-split with respect to $\Delta$, 
then there exists a $\mathbb{Z}_{(p)}$-Weil divisor $\Delta'\ge \Delta$ on $X$ such that $K_{X/Z}+\Delta'\sim_{\mathbb{Z}_{(p)}}0$. 
\begin{obs}\label{obs:CY}
Let $X$, $\Delta$, $Z$ and $f$ be as in Definition~\ref{defn:relFsp}. 
Assume that $(f,\Delta)$ is $F$-split. Then there exists $e \in \mathbb Z_{>0}$ such that 
$\phi^{(e)}_{(X/Z,\Delta)}:{F_{X/Z}^{(e)}}_* \mathcal L^{(e)}_{(X/Z,\Delta)}\to\O_{X_{Z^e}}$
splits as a homomorphism of $\O_{X_{Z^e}}$-module. Here, we recall that 
$$\mathcal L^{(e)}_{(X/Z,\Delta)}:=\O_{X^e}(\lfloor(1-p^e)(K_{X^e/Z^e}+\Delta)\rfloor).$$
Then there exists an element $s\in H^0(X^e,\lfloor(1-p^e)(K_{X^e/Z^e}+\Delta)\rfloor)$ 
such that $\phi^{(e)}_{(X/Z,\Delta)}$ sends $s$ to $1$. 
Let $E$ be an effective Weil divisor on $X^e$ defined by $s$. 
Set $\Delta':=(p^e-1)^{-1}\lceil(p^e-1)\Delta+E\rceil\ge\Delta$. 
Then by the choice of $E$ we have 
\begin{align*}
\mathcal L^{(e)}_{(X/Z,\Delta')}
:=&\O_{X^e}((1-p^e)(K_{X^e/Z^e}+\Delta')) \\
=&\O_{X^e}(\lfloor(1-p^e)(K_{X^e/Z^e}+\Delta)-E\rfloor)\cong\O_{X^e},
\end{align*}
and $\phi^{(e)}_{(X/Z,\Delta)}:{F_{X/Z}^{(e)}}_*\mathcal L^{(e)}_{(X/Z,\Delta')}\to\O_{X_{Z^e}}$ splits.
\end{obs}
The following lemma shows that an $F$-split morphism is surjective:
\begin{lem}\label{lem:surj}
Let $f:X\to Z$ be a projective morphism between normal varieties that is locally $F$-split. 
Then for each $i$, the sheaf $\mathcal G^i:=R^if_*\O_X$ is a vector bundle satisfying ${F_Z^e}^*\mathcal G^i\cong\mathcal G^i$ for some $e>0$. 
In particular, $f$ is surjective.
\end{lem}
\begin{proof} 
Applying the functor $R^i{f_{Z^e}}_*$ to $\O_{X_{Z^e}}\to{F_{X/Z}^{(e)}}_*\O_{X^e}$, 
we obtain the morphism $R^i{f_{Z^e}}_*\O_{X_{Z^e}}\to R^i{f^{(e)}}_*\O_{X^e}=\mathcal G^i$ 
which is injective and splits locally. 
Since $F_Z$ is flat, the sheaf $R^i{f_{Z^e}}_*\O_{X_{Z^e}}$ is isomorphic to ${F_Z^e}^*R^if_*\O_X={F_Z^e}^*\mathcal G^i$, 
so we get the morphism $\Phi^i:{F_Z^e}^*\mathcal G^i\to \mathcal G^i$. 
We can easily check that $\Phi^i$ is an isomorphism, 
and so the lemma below shows $\mathcal G^i$ is locally free. 
\end{proof}
\begin{lem}[\textup{\cite[Lemma 1.4]{MN84}}]\label{lem:vb}
Let $M$ be a finitely generated module over a regular local ring $R$ of positive characteristic. 
If ${F_R^e}^*M\cong M$ for some $e>0$, then $M$ is free.
\end{lem}
The next proposition shows that 
the local $F$-splitting of $(f,\Delta)$ requires certain conditions on $\Delta$ and the fibers of $f$.
\begin{prop}\label{prop:fiber}
Let $X$, $\Delta$, $Z$ and $f$ be as in Definition~\ref{defn:relFsp}. 
Assume that $(f,\Delta)$ is locally $F$-split and $\Delta$ is $\mathbb{Z}_{(p)}$-Cartier. 
Then the following hold: 
\begin{itemize}
\item[$(1)$] Let $f:X\xrightarrow{g}Y\xrightarrow{h}Z$ be the Stein factorization. Then $h$ is \'etale. 
\item[$(2)$] The support of $\Delta$ does not contain any irreducible component of any fiber. 
\item[$(3)$] For every scheme-theoretic point $z\in Z$, 
the pair $(X_{\ol z},\Delta_{\ol z})$ is $F$-split, 
where $X_{\ol z}$ is the geometric fiber over $z$. In particular, $X_{\ol z}$ is reduced.
\item[$(4)$] There exists an open subset $U\subseteq X$ such that 
$\mathrm{codim}(X\setminus U)\ge2$ and $f|_U:U\to Z$ is a smooth morphism.
In particular, the general fiber of $f$ is normal.
\end{itemize}
\end{prop}
\begin{proof}
We first prove (2) and (3). 
Take a point $z\in Z$. 
Restricting the homomorphism~(\ref{defn:relFsp}.1)$_e$ to $X_{\ol z^e}$, 
we obtain the homomorphism of $\O_{X_{\ol z^e}}$-modules
$$
\O_{X_{\ol z^e}}
\xrightarrow{{F_{X_{\ol z}/\ol z}^{(e)}}^\sharp} 
{F_{X_{\ol z}/\ol z}^{(e)}}_*\O_{(X_{\ol z})^e}
\to
{F_{X_{\ol z}/\ol z}^{(e)}}_* \left( \O_{(X_{\ol z})^e} \left( (p^e-1)\Delta|_{(X_{\ol z})^e} \right) \right)
$$ 
which is injective and splits for some $e>0$. 
This implies that the homomorphism 
$\O_{X_{\ol z}}\to(\O_X(p^e-1)\Delta)|_{X_{\ol z}}$ 
is not zero over each irreducible component. 
Hence, $\mathrm{Supp}\,\Delta$ does not contain any component of $X_{\ol z}$ 
and $(X_{\ol z},\Delta_{\ol z})$ is $F$-split, so (2) and (3) hold.
 
We show (4). Let $\pi:W\to X_{Z^e}$ be the normalization of $X_{Z^e}$. 
Then, $F_{X/Z}^{(e)}:X^e\to X_{Z^e}$ factors through $W$, and we have the morphisms 
$$
\O_{X_{Z^e}}\xrightarrow{\pi^\#}\pi_*\O_W\to{F_{X/Z}^{(e)}}_*\O_{X^e}
$$ 
of $\O_{X_{Z^e}}$-modules, which implies $\pi^\#$ is injective and splits. 
Since $\pi_*\O_W/\O_{X_{Z^e}}$ is a torsion $\O_{X_{Z^e}}$-module and $\pi_*\O_W$ is torsion-free, 
we find that $\pi_*\O_W/\O_{X_{Z^e}}=0$, so $\pi$ is an isomorphism, which means $X_{Z^e}$ is normal. 
Now, ${F_{X/Z}^{(e)}}_*\O_{X^e}$ is a torsion-free sheaf on a normal variety, 
so it is locally free over an open subset $U$ whose complement $X_{Z^e}\setminus U$ has codimension at least two.  
Therefore, for every $z\in Z$ we see that 
${F_{U_{\ol z}/{\ol z}}^{(e)}}_*\O_{(U_{\ol z})^e}\cong \left( {F_{U/Z}^{(e)}}_*\O_{U^e} \right) |_{U_{\ol z^e}}$ 
is locally free, and so $U_{\ol z}$ is regular by Kunz's theorem. 
Hence $f|_U:U\to Z$ is smooth.
This means that the general fiber satisfies $S_2$ and $R_1$, i.e. it is normal. 
 
We prove (1). Proposition~\ref{prop:fgh}~(1) shows that $h$ is $F$-split, 
so we see from Proposition~\ref{prop:fin} that $h$ is \'etale. 
\textbf{Note that (1) is used in Section~\ref{section:albf} but not in this section}, 
so the use of Propositions~\ref{prop:fin} and~\ref{prop:fgh} does not cause a logical problem. 
\end{proof}
On the contrary to Proposition~\ref{prop:fiber}~(3), the morphism $f$ is not necessarily $F$-split 
even if every fiber is $F$-split (see Theorem~\ref{thm:surf} for example). 
However, if $K_X$ is $\mathbb{Z}_{(p)}$-linearly trivial relative to $f$, 
then the converse holds as seen in Theorem~\ref{thm:cbf} below. 
This is used in the proofs of Proposition~\ref{prop:fin_alb} and Theorem~\ref{thm:surf}.
\begin{thm}[\textup{\cite[Theorem 3.17]{Eji17}}]\label{thm:cbf} 
Let $f:X\to Z$ be a surjective projective morphism from a normal variety $X$ 
to a smooth variety $Z$ such that $f_*\O_X\cong \O_Z$. 
Let $\Delta$ be an effective $\mathbb{Z}_{(p)}$-Weil divisor on $X$. 
Suppose that $K_X+\Delta\sim_{\mathbb{Z}_{(p)}} f^*C$ for some Cartier divisor $C$ on $Z$.  
Let $X_{\ol\eta}$ denote the geometric generic fiber of $f$. 
\begin{itemize}
\item[$(\rm i)$] If $(X_{\ol\eta},\Delta_{\ol\eta})$ is not $F$-split, 
then so is $(X_{\ol z},\Delta_{\ol z})$ for general $z\in Z$. 
\item[$(\rm ii)$] If $(X_{\ol\eta},\Delta_{\ol\eta})$ is $F$-split, 
then there exists an effective $\mathbb{Z}_{(p)}$-Weil divisor $\Delta_Z$ on $Z$ such that the following hold:
\begin{itemize}
\item[$(1)$] The divisor $(K_Z+\Delta_Z)$ is ${\mathbb{Z}_{(p)}}$-linearly equivalent to $C$.
\item[$(2)$] The pair $(X,\Delta)$ is $F$-split if and only if so is $(Z,\Delta_Z)$.
\item[$(3)$] The following are equivalent: 
\begin{itemize} 
\item[$(3$-$1)$] $(f,\Delta)$ is $F$-split;
\item[$(3$-$2)$] $(f,\Delta)$ is locally $F$-split;
\item[$(3$-$3)$] $(X_{\ol z},\Delta|_{X_{\ol z}})$ is $F$-split for every codimension one point $z\in Z$; 
\item[$(3$-$4)$] $\Delta_Z=0$.
\end{itemize}
Here, $X_{\ol z}$ is the geometric fiber over $z$.
\end{itemize} 
\end{itemize}
\end{thm}
\begin{proof}
Obviously,~(i) (resp.~(1)) follows from \cite[Observation 3.19]{Eji17} (resp. \cite[Theorem 3.17 (1)]{Eji17}). 
By \cite[Theorem 3.17 (2)]{Eji17}, we see that $S^0(X,\Delta,\O_X)\cong S^0(Z,\Delta_Z,\O_Z)$ 
(see \cite[\S 3]{Sch14} or \cite[Definition 3.2]{Eji17} for the definition of $S^0$), 
so~(2) is a consequence of the fact that 
$(X,\Delta)$ is $F$-split if and only if $S^0(X,\Delta,\O_X)=H^0(X,\O_X)$. 
To prove~(3), we recall the construction on $\Delta_Z$. 
Replacing $X$ and $Z$ by their smooth loci, we may assume that $X$ and $Z$ are smooth. 
For an $e>0$ with $a|(p^e-1)$, we have 
$${f^{(e)}}_*\mathcal L^{(e)}_{(X/Z,\Delta)}={f^{(e)}}_*\O_{X^e}((1-p^e)(K_{X^e/Z^e}+\Delta))\cong \O_{Z^e}((1-p^e)(C-K_{Z^e}))$$
by the projection formula. We then set 
$$
\theta^{(e)}:\O_{Z^e}((1-p^e)(C-K_{Z^e}))\cong{f^{(e)}}_*\mathcal L^{(e)}_{(X/Z,\Delta)} 
\xrightarrow{{f_{Z^e}}_*\phi^{(e)}_{(X/Z,\Delta)}} {f_{Z^e}}_*\O_{X_{Z^e}} 
\cong \O_{Z^e}.
$$
Since 
$
\left( {f_{Z^e}}_*\phi^{(e)}_{(X/Z,\Delta)} \right) \otimes k \left( \ol\eta^e \right)
\cong H^0 \left( X_{\ol\eta^e},\phi^{(e)}_{(X_{\ol\eta}/\ol\eta,\Delta_{\ol\eta})} \right)
$ 
is surjective because of the assumption, $\theta^{(e)}$ is non-zero. 
Hence there exists an effective divisor $E$ on $Z$ such that $\O_{Z^e}(-E)$ is equal to the image of $\theta^{(e)}$. 
Define $\Delta_Z:=(p^e-1)^{-1}E$. 
By definition, $\Delta_Z=0$ if and only if $\theta^{(e)}$ is surjective.
Furthermore, by the argument similar to the above, we see that for a codimension one point $z\in Z$, 
the pair $(X_{\ol z},\Delta|_{X_{\ol z}})$ is $F$-split 
if and only if $\theta^{(e)}\otimes k(\ol z)$ is non-zero, 
i.e. $\Delta$ is zero around $z$.
We now prove~(3). Obviously, (3-1)$\Rightarrow$(3-2). 
Proposition~\ref{prop:fiber}~(3) shows (3-2)$\Rightarrow$(3-3), 
and (3-3)$\Rightarrow$(3-4) follows from the above argument. 
If $\Delta_Z=0$, i.e., if $\theta^{(e)}$ is an isomorphism, 
then $H^0(X_{Z^e},\phi^{(e)}_{(X/Z,\Delta)})\cong H^0(Z^e,\theta^{(e)})$ is also an isomorphism,
so $\phi^{(e)}_{(X/Z,\Delta)}$ splits. This proves (3-4)$\Rightarrow$(3-1).
\end{proof}
Proposition~\ref{prop:fin} below deals with the case of finite morphisms. 
Note that, in the case when $\Delta=0$, Proposition~\ref{prop:fin} has already been proved in \cite[2.19 Theorem]{Has10}.
\begin{prop}\label{prop:fin}
Let $X$, $\Delta$, $Z$ and $f$ be as in Definition~\ref{defn:relFsp}. 
Assume that $\dim X=\dim Z$. Then the following are equivalent:
\begin{itemize}
\item[$(1)$] $(f,\Delta)$ is $F$-split;
\item[$(2)$] $(f,\Delta)$ is locally $F$-split;
\item[$(3)$] $f$ is \'etale and $\Delta=0$.
\end{itemize}
\end{prop}
\begin{proof} 
Obviously, (1)$\Rightarrow$(2). 
Let $f$ be \'etale and $\Delta=0$. 
Then $F_{X/Z}^{(e)}:X^e\to X_{Z^e}$ is a finite morphism of degree one between normal varieties, 
so it is an isomorphism, which implies (3)$\Rightarrow$(1). 
We show (2)$\Rightarrow$(3).  
By Lemma~\ref{lem:surj} and Proposition~\ref{prop:fiber}~(4), 
the morphism $f$ is surjective and separable,
so it follows from the assumption that $f$ is generically finite. 
Let $e>0$ be an integer such that the morphism 
$$
\O_{X_{Z^e}}\xrightarrow{\alpha}{F_{X/Z}^{(e)}}_*\O_{X^e}(\lceil(p^e-1)\Delta\rceil)
$$ splits.
Since $F_{X/Z}^{(e)}$ is a finite morphism of degree one, 
${F_{X/Z}^{(e)}}_*\O_{X^e}(\lceil(p^e-1)\Delta\rceil)$ is a torsion-free sheaf of rank one. 
Note that as $f$ is separable, $X_{Z^e}$ is a variety.
Therefore, $\mathrm{Coker}\,{\alpha}=0$, so it is an isomorphism, 
which means that $\Delta=0$ and $F_{X/Z}^{(e)}$ is an isomorphism.
Then, $F_{X_{\ol z}/\ol z}^{(e)}$ is also an isomorphism for every $z\in Z$, 
where $\ol z$ is the algebraic closure of $z$.
Hence, $X_{\ol z}$ is isomorphic to a disjoint union of copies of the spectrum of $k(\ol z)$. 
In particular, $f$ is finite.
The local freeness of $f_*\O_X$, which follows from Lemma~\ref{lem:surj}, implies $f$ is flat, 
and so we conclude that $f$ is \'etale.
\end{proof}
The following lemma is used in the proofs of Proposition~\ref{prop:fin_alb} and Theorem~\ref{thm:surf}.
\begin{lem}\label{lem:kappa}
Let $X$, $\Delta$, $Z$ and $f$ be as in Definition~\ref{defn:relFsp}. 
Assume that $(f,\Delta)$ is locally $F$-split and that $\Delta$ is a $\mathbb{Z}_{(p)}$-Weil divisor. 
Then the Iitaka--Kodaira dimension $\kappa(X,K_{X/Z}+\Delta)$ of $K_{X/Z}+\Delta$ is non-positive. 
Furthermore, if $(f,\Delta)$ is $F$-split, then $\kappa(X,-(K_{X/Z}+\Delta))\ge0$.
\end{lem}
\begin{proof}
The second statement follows from Observation~\ref{obs:CY}.
By Lemma~\ref{lem:surj}, the morphism $f$ is surjective. 
Assume that $\kappa(X,K_{X/Z}+\Delta)\ge0$. 
Then $\kappa(X_{\ol\eta},K_{X_{\ol\eta}/\ol\eta}+\Delta_{\ol\eta})\ge0$, 
where $\ol\eta$ is the geometric generic point of $Z$. 
Since $(X_{\ol\eta}, \Delta_{\ol \eta})$ is $F$-split, 
we have $H^0(X_{\ol\eta},(1-p^e)(K_{X_{\ol\eta}}+\Delta_{\ol\eta}))\ne0$ for some $e>0$, 
so $(1-p^e)(K_{X_{\ol\eta}}+\Delta_{\ol\eta})\sim 0$.
The morphism 
$$
{f_{Z^e}}_*\phi^{(e)}_{(X/Z,\Delta)}:{f^{(e)}}_*\O_{X^e}((1-p^e)(K_{X^e/Z^e}+\Delta))\to {f_{Z^e}}_*\O_{X_{Z^e}}
$$
is then a surjective morphism between torsion-free coherent sheaves of the same rank, 
and so it is an isomorphism.
Hence, $H^0(X,(1-p^e)(K_{X/Z}+\Delta))\ne0$, which implies $\kappa(X,K_{X/Z})=0$. 
This is our assertion.
\end{proof}
Finally, we consider the composition of (locally) $F$-split morphisms, which is used frequently in Section~\ref{section:albf}.
\begin{prop}\label{prop:fgh} 
Let $X$, $\Delta$, $Z$ and $f$ be as in Definition~\ref{defn:relFsp}, 
and $Y$ be a normal variety. 
Assume that $f:X\to Z$ can be factored into projective morphisms $g:X\to Y$ with $g_*\O_X\cong \O_Y$ and $h:Y\to Z$. 
Suppose that $Z$ is projective. 
\begin{itemize}
\item[$(1)$] If $(f,\Delta)$ is $F$-split, then so is $h$.
\item[$(2)$] Assume that $Y$ is smooth. If $(g,\Delta)$ and $h$ are $F$-split, then so is $(f,\Delta)$. 
\item[$(3)$] The converse of $(2)$ holds if $K_Y\sim_{\mathbb{Z}_{(p)}}h^*K_Z$. 
\end{itemize}
\end{prop} 
\begin{proof} 
Let $e>0$ be an integer. 
Now we have the following commutative diagram: 
$$
\xymatrix@R=25pt@C=50pt{ 
X^e \ar[r]|{F_{X/Y}^{(e)}} \ar@/^20pt/[rr]^{F_{X/Z}^{(e)}}\ar[dr]_{g^{(e)}} 
& {X_{Y^e}} \ar[d]^{g_{Y^e}} \ar[r]|{\pi^{(e)}} 
& X_{Z^e} \ar[r]^{(F_{Z}^{e})_X} \ar[d]^{g_{Z^e}} & X \ar[d]_{g} \ar@/^30pt/[dd]^{f} \\ 
 & Y^e \ar[r]^{F_{Y/Z}^{(e)}} \ar[dr]_{h^{(e)}} 
 & Y_{Z^e} \ar[d]^{h_{Z^e}} \ar[r]^{(F_{Z}^{e})_Y} & Y \ar[d]_{h} \\ 
 & & Z^e \ar[r]^{F_{Z}^e} & Z. 
}
$$ 
Here, $\pi^{(e)}:=(F_{Y/Z}^{(e)})_X$.  
We first show (1). The above diagram induces the commutative diagram of $\O_{Y_{Z^e}}$-modules 
$$
\xymatrix@R=15pt@C=20pt{ 
\O_{Y_{Z^e}} \ar[r] \ar[d]_{\cong} & {F_{Y/Z}^{(e)}}_*\O_{Y^e} \ar[d]^{\cong} \\ 
{g_{Z^e}}_*\O_{X_{Z^e}} \ar[r] & {g_{Z^e}}_*{F_{X/Z}^{(e)}}_*\O_{X^e}, 
}
$$ 
where the left vertical morphism is an isomorphism because of the flatness of $(F_Z^e)_Y$. 
Since the lower horizontal morphism splits, so does the upper one. 

Next, we show (2) and (3). 
As explained in Observation~\ref{obs:CY}, 
if $(g,\Delta)$ (resp. $(f,\Delta)$) is $F$-split, 
then there exists an effective $\mathbb{Z}_{(p)}$-Weil divisor $\Delta'\ge\Delta$ on $X$ 
such that $K_{X/Y}+\Delta'$ (resp. $K_{X/Z}+\Delta'$) is $\mathbb{Z}_{(p)}$-linearly trivial 
and that $(g,\Delta')$ (resp. $(f,\Delta')$) is also $F$-split.  
Therefore, we may assume that $\Delta$ is a $\mathbb{Z}_{(p)}$-Weil divisor 
and that $(p^e-1)(K_{X/Y}+\Delta)\sim0$ (resp. $(p^e-1)(K_{X/Y}+\Delta)\sim(p^e-1)(f^*K_Z-g^*K_Y)$) for every $e>0$ divisible enough. 
In particular, $\mathcal L^{(e)}_{(X/Y,\Delta)}$ (resp. $\mathcal L^{(e)}_{(X/Z,\Delta)}$) is isomorphic to 
the pullback by $g^{(e)}$ of a line bundle on $Y^{(e)}$. 

Let $V\subseteq Y$ be an open subset such that 
$X_V:=g^{-1}(V)$ is flat over $V$ and $\mathrm{codim}(Y\setminus V)\ge2$.  
Let $u:U\to X_V$ be the open immersion of the regular locus of $X_V$. 
Set $g':=g\circ u:U\to Y$.  
We then have ${g'}_*\O_U\cong g_*\O_X\cong \O_Y$ because of the assumptions.  
In addition, by the flatness of $F_Z^e$, we see that ${{g'}_{Z^e}}_*\O_{U_{Z^e}}\cong \O_{Y_{Z^e}}$.  
Hence, by the projection formula, we see that 
\begin{align*} 
H^0\left(U_{Z^e},({g_{Z^e}}^*\mathcal L)|_{U_{Z^e}}\right)
&\cong H^0\left(Y_{Z^e},{{g'}_{Z^e}}_*({{g'}_{Z^e}}^*\mathcal L)\right) 
\cong H^0\left(Y_{Z^e},\mathcal L\right)
\cong H^0\left(X_{Z^e},{g_{Z^e}}^*\mathcal L\right) 
\end{align*} 
for every line bundle $\mathcal L$ on $Y_{Z^e}$, 
and so we get the following commutative diagram: 
$$
\xymatrix@R=25pt@C=85pt{
H^0\left(U^e,\mathcal L^{(e)}_{(X/Z,\Delta)}|_{U^e}\right) \ar[r]^{H^0\left(U_{Z^e},\phi_{(U/Z,\Delta|_U)}^{(e)}\right)} \ar[d]_{\cong} 
& H^0\left(U_{Z^e},\O_{U_{Z^e}}\right) \ar[d]^{\cong} \\ 
H^0\left(X^e,\mathcal L^{(e)}_{(X/Z,\Delta)}\right) \ar[r]^{H^0\left(X_{Z^e}, \phi_{(X/Z,\Delta)}^{(e)}\right)} 
& H^0\left(X_{Z^e},\O_{X_{Z^e}}\right). 
}
$$ 
In particular, $H^0(U_{Y^e},\O_{U_{Y^e}})\cong H^0(Y^e,\O_{Y^e})\cong k$.  
Since the splitting of $\phi^{(e)}_{(X/Z,\Delta)}$ is clearly equivalent to 
the surjectivity of $H^0\left(X_{Z^e},\phi_{(X/Z,\Delta)}^{(e)}\right)$, 
we see that the $F$-splitting of $(f,\Delta)$ and that of $(f|_U:U\to Z,\Delta|_U)$ are equivalent.  
By an argument similar to the above, 
we find that the $F$-splitting of $(g,\Delta)$ and that of $(g|_U,\Delta|_U)$ are also equivalent. 

Assume that we can choose $V=Y$ and $U=X$, i.e. $X$ and $Y$ are regular and $g$ is flat.  
Let $e>0$ be an integer. 
By the flatness of $g$, we have the following commutative diagram: 
$$
\xymatrix@R=25pt@C=85pt{ 
{{g}_{Z^e}}^*\O_{Y_{Z^e}} \ar[r]^{ {{g}_{Z^e}}^*\left({F_{Y/Z}^{(e)}}^{\sharp}\right)} \ar[d]_{\cong} 
& {{g}_{Z^e}}^*{F_{Y/Z}^{(e)}}_*\O_{Y^e} \ar[d]^{\cong} \\ 
\O_{X_{Z^e}}\ar[r]^{{\pi^{(e)}}^{\sharp}} & {\pi^{(e)}}_*\O_{X_{Y^e}}.
} 
$$ 
This implies that 
$$
\mathcal Hom \left({\pi^{(e)}}^{\sharp},\O_{X_{Z^e}}\right) 
\cong {{g}_{Z^e}}^*\mathcal Hom \left({F_{Y/Z}^{(e)}}^{\sharp},\O_{V_{Z^e}}\right)
={{g}_{Z^e}}^*\phi_{Y/Z}^{(e)}.
$$ 
Applying the functor $\mathcal Hom(\underline{\quad},\O_{X_{Z^e}})$ 
and the Grothendieck duality to the natural morphism 
$$
\O_{X_{Z^e}}
\xrightarrow{{\pi^{(e)}}^{\sharp}} {\pi^{(e)}}_*\O_{X_{Y^e}}
\to {F_{X/Z}^{(e)}}_*\O_{X^e}(\lceil(p^e-1)\Delta\rceil),
$$ 
we obtain the morphism 
$$
\phi^{(e)}_{(X/Z,\Delta)}:{F_{X/Z}^{(e)}}_*\mathcal L^{(e)}_{(X/Z,\Delta)}
\xrightarrow{{\pi^{(e)}}_*\phi^{(e)}_{(X/Y,\Delta)}
\otimes\o_{\pi^{(e)}}} {g_{Z^e}}^*{F_{Y/Z}^{(e)}}_*\mathcal L^{(e)}_{Y/Z}
\xrightarrow{{g_{Z^e}}^*\phi^{(e)}_{Y/Z}} \O_{X_{Z^e}}.
$$ 
Note that 
$$ 
\o_{\pi^{(e)}}
\cong\o_{X_{Y^e}}\otimes{\pi^{(e)}}^*\o_{X_{Z^e}} 
\cong {g_{Z^e}}^*\o_{Y^e/Z^e}^{1-p^e} 
\quad\textup{and}\quad
{g_{Z^e}}^*{F_{Y/Z}^{(e)}}_*\mathcal L^{(e)}_{Y/Z} 
\cong {\pi^{(e)}}_*{g_{Y^e}}^*\mathcal L^{(e)}_{Y/Z}.  
$$ 
We now prove the assertion. 
If $(g,\Delta)$ is $F$-split and $h$ is $F$-split, 
then both of $\phi^{(e)}_{(X/Y,\Delta)}$ and $\phi^{(e)}_{Y/Z}$ split for every $e>0$ divisible enough. 
Therefore, $\phi^{(e)}_{(X/Z,\Delta)}$ also splits, i.e. $(f,\Delta)$ is $F$-split.  
Conversely, suppose that $(f,\Delta)$ is $F$-split and that $(p^e-1)K_{Y/Z}\sim 0$ for an $e>0$.  
Then, $\mathcal L^{(e)}_{Y/Z}\cong\O_{Y_{Z^e}}$ and $\o_{\pi^{(e)}}\cong\O_{X_{Y^e}}$. 
Fix an $e>0$ divisible enough. 
Since $H^0\left(X_{Z^e},\phi_{(X/Z,\Delta)}^{(e)}\right)$ is surjective, 
$H^0\left(X_{Z^e},{\pi^{(e)}}_*\phi_{(X/Y,\Delta)}^{(e)}\right)$ is a non-zero morphism, 
and hence so is $H^0\left(X_{Y^e},\phi^{(e)}_{(X/Y,\Delta)}\right)$. 
This morphism is surjective because of $H^0(X_{Y^e},\O_{X_{Y^e}})\cong H^0(Y^e,\O_{Y^e})\cong k$. 
Thus, $\phi^{(e)}_{(X/Y,\Delta)}$ splits, and so $(g,\Delta)$ is $F$-split.  
Note that the $F$-splitting of $h$ follows directly from (1).  
\end{proof}
\section{Varieties with $F$-split Albanese morphism} \label{section:albf}
In this section, we prove 
Theorems~\ref{thm:ch_Fsp_intro}, \ref{thm:albf_intro}, \ref{thm:ch_ab_intro} and \ref{thm:nscond}.
Throughout this section, we fix an algebraically closed field $k$ of characteristic $p>0$, 
and we denote by $X$ and $\Delta$ respectively 
a normal projective variety over $k$ and an effective $\mathbb{Q}$-Weil divisor on $X$. 
\begin{proof}[Proof of Theorem \ref{thm:albf_intro}]
Suppose that $(a,\Delta)$ is locally $F$-split, 
and let $X\xrightarrow{f}Z\xrightarrow{g}A$ be the Stein factorization of $a$. 
Thanks to Proposition~\ref{prop:fiber}, 
we only need to show that the \'etale morphism $g$ is an isomorphism. 
Applying~\cite[Section 18, Theorem]{Mum70}, we see that $Z$ is an abelian variety, 
so the universal property of $a$ tells us that $g$ is an isomorphism, which completes the proof. 
\end{proof}
The next lemma is used to prove Theorems~\ref{thm:ch_Fsp_intro} and~\ref{thm:ch_ab_intro}.
\begin{lem}\label{lem:inv_lb} 
Let $\mathcal F$ be a coherent sheaf of rank $r$ on a normal variety $Y$. 
Let $\mathcal F'$ be an indecomposable direct summand of $\mathcal F$ whose rank is $r'$. 
Set $I:=\{\mathcal L\in\mathrm{Pic}(Y)|\mathcal F\otimes\mathcal L\cong\mathcal F\}$ and $I':=\{\mathcal L\in I|\mathcal F'\otimes\mathcal L\cong\mathcal F'\}$. 
Then $\bigoplus_{[\mathcal L]\in I/I'}\mathcal F'\otimes\mathcal L$ is a direct summand of $\mathcal F$. 
In particular, $\#(I/I')\le r/r'$.  
\end{lem} 
\begin{proof} 
For every $\mathcal L\in I$, the sheaf $\mathcal F'\otimes\mathcal L$ is also a direct summand of $\mathcal F$. 
Furthermore, $\mathcal F\otimes\mathcal L\cong\mathcal F\otimes\mathcal L'$ if and only if $\mathcal L'\otimes\mathcal L^{-1}\in I$. 
Hence, the Krull--Schmidt theorem~\cite{Ati56} tells us that 
$\bigoplus_{[\mathcal L]\in I/I'}\mathcal F'\otimes\mathcal L$ is a direct summand of $\mathcal F$, 
and so $r'\#(I/I')\le r$, which is our claim.  
\end{proof}
To prove Theorem \ref{thm:ch_Fsp_intro}, 
we recall a characterization of ordinary abelian varieties which was established by Sannai and Tanaka. 
\begin{thm}[\textup{\cite[Theorem 1.1]{ST16}}]\label{thm:ST}
A smooth projective variety $Y$ is an ordinary abelian variety 
if and only if $K_Y$ is pseudo-effective and ${F_Y^e}_*\O_Y$ is isomorphic to a direct sum of line bundles for infinitely many $e>0$.
\end{thm}
\begin{rem}\label{rem:ES}
It was shown in \cite{ES19} that 
we actually need to check the decomposition of ${F_Y^e}_*\O_X$ only for $e=2$ in the above theorem. 
\end{rem}
For convenience, we use the following notation: 
\begin{notation}\label{notation:kernel}
Let $\varphi:S\to T$ be a morphism of schemes. 
We denote by $\mathrm{Pic}(S)[\varphi]$ $($resp. $\mathrm{Pic}^0(S)[\varphi]$$)$ 
the kernel of the induced homomorphism $\varphi^*:\mathrm{Pic}(T)\to\mathrm{Pic}(S)$ $($resp. $\varphi^*:\mathrm{Pic}^0(T)\to\mathrm{Pic}^0(S)$$)$.
Then, $\mathrm{Pic}(X)[{F_X}^e]$ is the set of $p^e$-torsion line bundles for every $e>0$. 
We denote it by $\mathrm{Pic}(X)[p^e]$.
\end{notation}
\begin{proof}[Proof of Theorem \ref{thm:ch_Fsp_intro}]
We first prove that if $(X,\Delta)$ is $F$-split, 
then $(a,\Delta)$ is $F$-split and $A$ is ordinary.
We have the following commutative diagram 
$$ 
\xymatrix@R=20pt @C=20pt{ 
H^1(X,\O_X) \ar[r]^{{F_X}^*} & H^1(X,\O_X) \\ 
H^1(A,\O_A) \ar[u]^{a^*} \ar[r]^{{F_A}^*} & H^1(A,\O_A) \ar[u]_{a^*}. 
}
$$ 
Since $X$ is $F$-split, the upper horizontal arrow is bijective. 
By \cite[Lemma(1.3)]{MS87}, we see that the vertical arrows are injective, 
so the lower horizontal arrow is also injective, which implies $A$ is ordinary. 
(Note that, although $X$ is assumed to be smooth in \cite[Lemma(1.3)]{MS87}, 
the proof does not use smoothness of $X$.) 
 
Let $X\xrightarrow{f}Z\xrightarrow{g}A$ be the Stein factorization of $a$. 
Proposition~\ref{prop:fgh}~(1) then shows that $Z$ is $F$-split, 
so $\O_Z$ is a direct summand of $\mathcal F^{(e)}:={F_Z^{e}}_*\O_{Z^e}$ for each $e>0$. 
Since $a^*:\mathrm{Pic}^0(A)\to\mathrm{Pic}^0(X)$ is bijective, 
$g^*:\mathrm{Pic}^0(A)\to\mathrm{Pic}^0(Z)$ is injective, so 
$$
p^{e\cdot\dim A}=\#\mathrm{Pic}^0(A)[F_A^e]\le\#\mathrm{Pic}^0(Z)[F_Z^e].
$$ 
By the projection formula and Lemma \ref{lem:inv_lb} (set $\mathcal F:=\mathcal F^{(e)}$ and $\mathcal F':=\O_Z$), 
we obtain
$$
p^{e\cdot\dim A}
\le \#\{\mathcal L\in\mathrm{Pic}(Z)|\mathcal F^{(e)}\otimes\mathcal L\cong\mathcal F^{(e)}\}
\le {\rm rank}\;\mathcal F^{(e)}
=p^{e\cdot\dim Z}.
$$
This implies that $\dim Z=\dim A$ and that 
$\bigoplus_{\mathcal L\in\mathrm{Pic}(Z)[p^e]}\mathcal L\subseteq\mathcal F^{(e)}$ is a direct summand of maximum rank. 
This inclusion is an isomorphism because of the torsion-freeness of $\mathcal F^{(e)}$. 
In particular, $F_Z^e$ is flat, i.e. $Z$ is smooth. 

We now only need to prove that $K_Z$ is pseudo-effective. 
If this holds, then Theorem~\ref{thm:ST} shows that $Z$ is an ordinary abelian variety, 
so we see from Proposition~\ref{prop:fgh}~(3) that $(a,\Delta)$ is $F$-split, which is our assertion. 
Fix $e>0$. We now have $(\mathcal F^{(e)})^*\cong\mathcal F^{(e)}$ and ${F_Z^e}^*\mathcal F^{(e)}\cong\bigoplus\O_{Z^e}$. 
By (\ref{subsection:trace}.2) of Subsection \ref{subsection:trace}, we have $$
{F_Z^e}_*\o_{Z^e}^{1-p^e}\cong \mathcal Hom({F_Z^e}_*\O_{Z^e},\O_{Z})=(\mathcal F^{(e)})^*\cong \mathcal F^{(e)}, 
$$
so we have surjective morphisms
$\bigoplus\O_{Z^e}\cong{F_Z^e}^*\mathcal F^{(e)}\cong{F_Z^e}^*{F_Z^e}_*\o_{Z^e}^{1-p^e}\to\o_{Z^e}^{1-p^e}$, 
which implies that $\o_{Z}^{1-p^e}$ is globally generated. 
Since $H^0(Z^e,\o_{Z^e}^{1-p^e})\cong H^0(Z,\mathcal F^{(e)})\cong k$, 
we get $\o_{Z}^{1-p^e}\cong\O_Z$, i.e. $\o_{Z}^{p^e-1}\cong\O_Z$, and so $K_Z$ is pseudo-effective. 

The converse follows directly from Proposition \ref{prop:fgh}.
\end{proof}
\begin{proof}[Proof of Theorem \ref{thm:ch_ab_intro}]
Assume that $(a,\Delta)$ is locally $F$-split. 
Lemma~\ref{lem:surj} shows the first statement. 
We show the second. Suppose $\dim A=\dim X$. 
Then Proposition~\ref{prop:fin} tells us that $a$ is an isomorphism and $\Delta=0$. 
\end{proof}
The remainder of this section is devoted to prove Theorem~\ref{thm:nscond} below. 
We continue to use the same notation as that introduced at the begging of this section. 
\begin{thm}\label{thm:nscond} 
Let $\gamma_A$ denote the $p$-rank of $A$. 
Let $f:X\to B$ be a morphism to an abelian variety $B$ of $p$-rank $\gamma_B$. 
Suppose that $(f,\Delta)$ is $F$-split. 
Then $(a,\Delta)$ is $F$-split and $\gamma_A=\gamma_B+\dim A-\dim B$. 
In particular, if $B$ is ordinary, then $(X,\Delta)$ is $F$-split.  
\end{thm}
This theorem is a consequence of Proposition~\ref{prop:fin_alb}. 
To prove the proposition, we need the lemma below. 
\begin{lem}\label{lem:h1}
Let $f:X\to Z$ be an $F$-split morphism to a smooth projective variety $Z$. 
Let $X_z$ be the general fiber of $f$. 
Then $$h^1(X,\O_X)\le h^1(X_z,\O_{X_z})+h^1(Z,\O_Z).$$
\end{lem}
\begin{proof}
Set $\mathcal G^i:=R^if_*\O_X$. 
Lemma~\ref{lem:surj} then shows that 
${\rm rank}\;\mathcal G^i=h^i(X_z,\O_{X_z})$ and ${F_Z^e}^*\mathcal G^i\cong \mathcal G^i$ for some $e>0$. 
Applying \cite[1.4. Satz]{LS77}, 
we find an \'etale cover $\pi:Z'\to Z$ such that $\pi^*\mathcal G^i\cong\bigoplus\O_{Z'}$ for each $i$, 
so
$
h^0(Z,\mathcal G^i) \le h^0(Z',\pi^*\mathcal G^i)=\mathrm{rank}\;\mathcal G^i=h^i(X_z,\O_{X_z}).
$
By the Leray spectral sequence, we conclude that
$$
h^1(X,\O_X)\le h^0(Z,\mathcal G^1)+h^1(Z,\O_Z) \le h^1(X_z,\O_{X_z})+h^1(Z,\O_Z).
$$
\end{proof}
\begin{obs}\label{obs:p-tors}
Let $f:X\to Z$ be a separable surjective morphism to a smooth projective variety $Z$ such that $f_*\O_X\cong \O_Z$. \\
(1) We consider the following commutative diagram: 
$$
\xymatrix@R=25pt@C=25pt{
\mathrm{Pic}(X^e) & \\  
\mathrm{Pic}(X_{Z^e}) \ar[u]^{{F_{X/Z}^{(e)}}^*} & \mathrm{Pic}(X) \ar[l]^{{(F_Z^e)_X}^*} \ar[ul]_{{F_X^e}^*} \\ 
\mathrm{Pic}(Z^e) \ar[u]^{{f_{Z^e}}^*} & \mathrm{Pic}(Z) \ar[l]_{{F_Z^e}^*} \ar[u]_{f^*} 
} 
$$
Clearly, $f^*$ induces an injective homomorphism $\mathrm{Pic}(Z)[p^e]\xrightarrow{f^*}\mathrm{Pic}(X)[(F_Z^e)_X]$. 
This is actually an isomorphism. 
Indeed, for every $\mathcal L\in\mathrm{Pic}(X)[(F_Z^e)_X]$, we see from the flatness of $F_Z^e$ that 
$$
{F_Z^e}^*f_*\mathcal L
\cong{f_{Z^e}}_*\mathcal L_{Z^e}
\cong{f_{Z^e}}_*\O_{X_{Z^e}}
\cong{F_Z^e}^*f_*\O_X
\cong\O_{Z^e}, 
$$
so $f_*\mathcal L$ is a $p^e$-torsion line bundle 
such that the natural morphism $f^*f_*\mathcal L\to\mathcal L$ is an isomorphism. 
Hence, we now have the exact sequence
$$
0\to \mathrm{Pic}(Z)[p^e]
\xrightarrow{f^*}\mathrm{Pic}(X)[p^e]
\xrightarrow{{(F_Z^e)_X}^*} \mathrm{Pic}(X_{Z^e})[F_{X/Z}^{(e)}].
$$
\noindent(2) 
Set $\mathcal F:={F_{X/Z}^{(e)}}_*\O_{X^e}$ and $I:=\{\mathcal L\in\mathrm{Pic}(X_{Z^e})|\mathcal F\otimes\mathcal L\cong\mathcal F\}$. 
Then, we have $\mathrm{Pic}(X_{Z^e})[F_{X/Z}^{(e)}]\subseteq I$ by the projection formula. 
Let $\mathcal F'$ be an indecomposable direct summand of $\mathcal F$ and let $I':=\{\mathcal L\in\mathrm{Pic}(X_{Z^e})|\mathcal F'\otimes\mathcal L\cong\mathcal F'\}$.
Then by Lemma \ref{lem:inv_lb}, we obtain that $\bigoplus_{[\mathcal L]\in I/I'}\mathcal F'\otimes\mathcal L$ is a direct summand of $\mathcal F$. 
In particular, $${{\rm rank}\;\mathcal F'}\cdot\#(I/I')\le{{\rm rank}\;}\mathcal F=p^{e(\dim X-\dim Z)}.$$
\end{obs}
\begin{prop}\label{prop:fin_alb}
Let $f:X\to Z$ be an $F$-split morphism to an abelian variety $Z$. 
Suppose that the Albanese morphism $a:X\to A$ of $X$ is a finite morphism. 
Then, $a$ is an isomorphism, or equivalently, $X$ is an abelian variety.
\end{prop}
\begin{proof}
Let $f:X\xrightarrow{f'}Z'\xrightarrow{\pi}Z$ be the Stein factorization. 
Then, Proposition~\ref{prop:fiber}~(1) shows that $\pi$ is \'etale, 
so $Z'$ is an abelian variety by \cite[Section 18, Theorem]{Mum70}. 
Combining this with Proposition~\ref{prop:fgh}~(3), 
we see that $(f',\Delta)$ is $F$-split. 
We may assume that $f_*\O_X\cong \O_Z$ by replacing $Z$ by $Z'$. 
By the universal property of $a$, we can write $f:X\xrightarrow{a}A\xrightarrow{g}Z$. 
Let $z\in Z$ be a general point. 
Proposition~\ref{prop:fiber} then tells us that $X_z$ is integral, normal and $F$-split.
Recall that $a$ is assumed to be finite. 
The induced morphism $X_z\to (A_z)_\mathrm{red}$ is then a finite morphism to an abelian variety, 
so Theorem~\ref{thm:ch_ab_intro} implies that $X_z$ is an ordinary abelian variety. 
Therefore, by Lemma~\ref{lem:h1}, we have 
$$ 
\dim A\le h^1(X,\O_X)\le h^1(X_z,\O_{X_z})+h^1(Z,\O_Z)=\dim X_z+\dim Z=\dim X, 
$$
and so $a$ is a surjective finite morphism. 
Since $f$ is $F$-split, it is separable, and hence so is $g$, which implies that $A_z$ is reduced.
We may assume $X_z\to A_z$ is an isogeny of abelian varieties. 
Considering $p$-torsion points, we find that $A_z$ is also ordinary, 
so Theorem~\ref{thm:cbf}~(ii) says that $g$ is $F$-split, and 
\begin{align*}
\gamma_A=\gamma_{A_z}+\gamma_Z
=\dim A_z+\gamma_Z
=\dim A-\dim Z+\gamma_Z, 
\tag{\ref{prop:fin_alb}.1} \label{eq:p-rank}
\end{align*}
where $\gamma$ denotes the $p$-rank. 
\begin{cl}\label{cl:sep}
The morphism $a:X\to A$ is separable.
\end{cl}
If the claim holds, then $X$ is an abelian variety. 
Indeed, since $f$ is $F$-split, Lemma~\ref{lem:kappa} says that 
$$
\kappa(X,K_{X/A})=\kappa(X,K_{X})=\kappa(X,K_{X/Z})\le0, 
$$ 
so the Zariski--Nagata purity theorem implies $a$ is \'etale, 
and hence it follows from \cite[Section 18, Theorem]{Mum70} that $X$ is an abelian variety. 
\end{proof}
\begin{proof}[Proof of Claim \ref{cl:sep}]
Let $s:Y\to A$ be the normalization of $A$ in the separable closure of $K(X)/K(A)$. 
Then, there is the purely inseparable finite cover $i:X\to Y$ with $a=s\circ i$. 
We prove $i$ is an isomorphism. 
There exist an $e>0$ and a morphism $b:Y^e\to X$ such that the following diagram commutes: 
$$
\xymatrix@R=25pt@C=25pt{
X^e \ar[r]^{F_X^{e}} \ar[d]_{i^{(e)}} & X \ar[d]^{i} \\ 
Y^e \ar[r]_{F_Y^e} \ar[ur]^{b} & Y.
}
$$
This induces the following commutative diagram:
$$
\xymatrix@R=25pt@C=25pt{
X^e \ar[r]^{F_{X/Z}^{(e)}} \ar[d]_{i^{(e)}} & X_{Z^e} \ar[d]^{i_{Z^e}} \ar[r]^{(F_Z^e)_X} & X 
\ar[d]^i \ar@/^25pt/[dd]_a \ar@/^50pt/[ddd]^f\\ 
Y^e \ar[r]_{F_{Y/Z}^{(e)}} \ar[ur]^{b_{Z^e}} & Y_{Z^e} \ar[r]_{(F_Z^e)_Y} \ar[d]|{s_{Z^e}} & Y \ar[d]^s \\
 & A_{Z^e} \ar[r]_{(F_Z^e)_A} \ar[d]_{g_{Z^e}} & A \ar[d]^g \\
 & Z^e \ar[r]_{F_Z^e} & Z. 
}
$$
Note that since $f$ and $g\circ s$ are separable, 
Lemma \ref{lem:bc} says that $X_{Z^e}$ and $Y_{Z^e}$ are varieties. 
The splitting of ${F_{X/Z}^{(e)}}^\#:\O_{X_{Z^e}}\to {F_{X/Z}^{(e)}}_*\O_{X^e}$ induces
that of ${b_{Z^e}}^\#:\O_{X_{Z^e}}\to {b_{Z^e}}_*\O_{Y^e}$. 
Pushing forward via $i_{Z^e}$, we find that 
${i_{Z^e}}_*\O_{X_{Z^e}}$ is isomorphic to a direct summand of $\mathcal F:={F_{Y/Z}^{(e)}}_*\O_{Y^e}$. 
Let $\mathcal F'$ be the indecomposable direct summand of ${i_{Z^e}}_*\O_{X_{Z^e}}$ with $H^0(Y_{Z^e},\mathcal F')\ne0$.
Set 
$$
I:=\{\mathcal L\in\mathrm{Pic}(Y_{Z^e})|\mathcal F\otimes\mathcal L\cong\mathcal F\}\textup{~and~}I':=\{\mathcal L\in I|\mathcal F'\otimes\mathcal L\cong\mathcal F'\}.
$$ 
For any $p^e$-torsion line bundle $\mathcal L$ on $Y$, we have $\mathcal L_{Z^e}={(F_Z^e)_Y}^*\mathcal L\in I$. 
Indeed, 
\begin{align*}
\mathcal F\otimes\mathcal L_{Z^e} &= \left({F_{Y/Z}^{(e)}}_*\O_{Y^e}\right)\otimes_{\O_{Y_{Z^e}}}\mathcal L_{Z^e} \\
&\cong {F_{Y/Z}^{(e)}}_*\left({F_{Y/Z}^{(e)}}^*\mathcal L_{Z^e}\right) 
\hspace{30pt}\textup{{\tiny projection formula}}\\
&\cong {F_{Y/Z}^{(e)}}_*\left({F_{Y}^e}^*\mathcal L\right) 
\hspace{55pt}\textup{{\tiny ${F_{Y/Z}^{(e)}}^*\circ{(F_Z^e)_Y}^*={F_Y^e}^*$}}\\
&\cong {F_{Y/Z}^{(e)}}_*\O_{Y^e}
\hspace{80pt}\textup{{\tiny $\mathcal L$ is $p^e$-torsion}}\\
&=\mathcal F
\hspace{120pt}\textup{{\tiny definition.}}
\end{align*}
Consider the morphism $\Phi:\mathrm{Pic}(A)[p^e]\to I/I'$ obtained by 
\begin{align*}
\mathrm{Pic}(A)[p^e]
\xrightarrow{s^*} \mathrm{Pic}(Y)[p^e]
\xrightarrow{{(F_Z^e)_Y}^*} I 
\xrightarrow{\textup{natural}} I/I'. 
\end{align*}
We show that the following sequence is exact: 
\begin{align*}
0\to \mathrm{Pic}(Z)[p^e]\xrightarrow{g^*}\mathrm{Pic}(A)[p^e]\xrightarrow{\Phi} I/I'.  
\tag{\ref{cl:sep}.2} \label{seq:p-tors}
\end{align*}
Take $\mathcal M\in\mathrm{Pic}(Z)[p^e]$. 
Then, 
$$
{(F_{Z}^e)_Y}^*s^*g^*\mathcal M
\cong{s_{Z^e}}^*{g_{Z^e}}^*{F_Z^e}^*\mathcal M
\cong{s_{Z^e}}^*{g_{Z^e}}^*\O_{Z^e}
\cong\O_{Y_{Z^e}}
\in I'.
$$ 
Take $\mathcal M\in\mathrm{Pic}(A)[p^e]$ such that $\mathcal N:={(F_Z^e)_Y}^*s^*\mathcal M\in I'$. 
We then have 
\begin{align*}
0\ne H^0(Y_{Z^e},\mathcal F')&\cong H^0(Y_{Z^e},\mathcal F'\otimes\mathcal N) 
\hspace{70pt}\textup{{\tiny $\mathcal N\in I'$}}\\
&\subseteq H^0(Y_{Z^e},({i_{Z^e}}_*\O_{X_{Z^e}})\otimes\mathcal N)
\hspace{30pt}\textup{{\tiny definition of $\mathcal F'$}}\\
&\cong H^0(X_{Z^e},{i_{Z^e}}^*\mathcal N)
\hspace{75pt}\textup{{\tiny projection formula,}} 
\end{align*}
so we get ${i_{Z^e}}^*\mathcal N\cong\O_{X_{Z^e}}$, which implies that  
$$
{(F_Z^e)_X}^*a^*\mathcal M
\cong {i_{Z^e}}^*{(F_Z^e)_Y}^*s^*\mathcal M
={i_{Z^e}}^*\mathcal N
\cong\O_{X_{Z^e}}. 
$$
Observation \ref{obs:p-tors} (1) tells us that 
$ a^*\mathcal M \cong a^*g^*\mathcal P $ for some $\mathcal P\in\mathrm{Pic}(Z)[p^e]$, 
and the injectivity of $a^*:\mathrm{Pic}^0(A)\to\mathrm{Pic}^0(X)$ shows that $\mathcal M\cong g^*\mathcal P$. 
Hence, we see that the sequence~\ref{seq:p-tors} is exact. 
 
Now, we have the following inequalities: 
\begin{align*}
p^{e(\dim A-\dim Z)} \cdot \mathrm{rank}\,\mathcal F'
&= p^{e(\gamma_A-\gamma_Z)} \cdot \mathrm{rank}\,\mathcal F'
\hspace{45pt}\textup{{\tiny by (\ref{eq:p-rank})}} \\
&\le \#(I/I') \cdot \mathrm{rank}\,\mathcal F'
\hspace{50pt}\textup{{\tiny by (\ref{seq:p-tors})}} \\
&\le \mathrm{rank}\,\mathcal F=p^{e(\dim A-\dim Z)}
\hspace{20pt}\textup{{\tiny Observation~\ref{obs:p-tors}~(2)}}  
\end{align*}
This implies that $\mathrm{rank}\,\mathcal F'=1$ and 
$\#(I/I')=p^{e(\dim A-\dim Z)}$, so $\Phi$ is surjective. 
Furthermore, it follows from the torsion-freeness of $\mathcal F$ that 
$\mathcal F\cong \bigoplus_{[\mathcal L]\in I/I'}\mathcal F'\otimes\mathcal L$. 
Since ${i_{Z^e}}_*\O_{X_{Z^e}}$ is a direct summand of $\mathcal F$, 
it is isomorphic to $\bigoplus_{[\mathcal L]\in J}\mathcal F'\otimes\mathcal L$ 
for some $J\subseteq I/I'$. 
However, we can prove that this $J$ must be $\{[\O_{Y_{Z^e}}]\}$ as follows. 
Take $\mathcal M\in\mathrm{Pic}(A)[p^e]$ with $\Phi(\mathcal M)\in J$. 
Set $\mathcal N:={(F_Z^e)_Y}^*s^*\mathcal M$. 
We then have 
\begin{align*}
0\ne H^0(Y_{Z^e},\mathcal F')
&=H^0\left(Y_{Z^e},\mathcal F'\otimes\mathcal N\otimes\mathcal N^{-1}\right) \\
& \subseteq H^0\left(Y_{Z^e},({i_{Z^e}}_*\O_{X_{Z^e}})\otimes\mathcal N^{-1}\right)
=H^0(X_{Z^e},{i_{Z^e}}^*\mathcal N^{-1}), 
\end{align*}
so ${(F_Z^e)_X}^*a^*\mathcal M\cong{i_{Z^e}}^*\mathcal N\cong \O_{X_{Z^e}}$. 
By an argument similar to the above, 
we find some $\mathcal P\in\mathrm{Pic}(Z)[p^e]$ with $\mathcal M\cong g^*\mathcal P$, 
so the exactness of (\ref{seq:p-tors}) shows that $\Phi(\mathcal M)=[\O_{Y_{Z^e}}]$. 
We then see that $J=\{[\O_{Y_{Z^e}}]\}$ by the surjectivity of $\Phi$, 
so $i_{Z^e}$ is an isomorphism, and hence so is $i$, which proves our claim. 
\end{proof}
\begin{proof}[Proof of Theorem \ref{thm:nscond}]
Let $X\xrightarrow{\pi}X'\xrightarrow{g'}A$ be the Stein factorization of $a$. 
Then, $f$ can be written as $f:X\xrightarrow{\pi}X'\xrightarrow{g'}A\xrightarrow{h} B$. 
Proposition~\ref{prop:fgh}~(1) says that $h\circ g':X'\to B$ is $F$-split. 
Since the morphism $g':X'\to A$ is finite and is the Albanese morphism of $X'$, 
we see from Proposition~\ref{prop:fin_alb} that $g'$ is an isomorphism. 
Therefore, Proposition~\ref{prop:fgh}~(3) implies that $(a,\Delta)$ is $F$-split. 
Since $h:A\to B$ is an $F$-split morphism whose closed fibers $A_z$ are ordinary abelian varieties, we obtain
$$\gamma_A=\gamma_{A_z}+\gamma_B=\dim A_z+\gamma_B=\dim A-\dim B+\gamma_B.$$
\end{proof}
\section{Minimal surfaces with $F$-split Albanese morphism}\label{section:surf}
Fix an algebraically closed field $k$ of characteristic $p>0$. 
In this section, we study the Albanese morphism 
of a minimal surface over $k$ in terms of $F$-splitting. 
Here, by a minimal surface, we mean a smooth projective surface with no (-1)-curves. 
Note that it follows from Proposition~\ref{prop:fgh}~(1) that 
if a smooth projective surface has $F$-split (resp. locally $F$-split) Albanese morphism, 
then so does its minimal surface.  
Throughout this section, $X$ denotes a smooth projective minimal surface 
and $a:X\to A$ denotes the Albanese morphism of $X$.
\begin{thm}\label{thm:surf}
The morphism $a$ is locally $F$-split if and only if 
one of the following conditions holds:
\begin{itemize}
\item[$(0)$\hspace{10pt}] $b_1(X)=0$ and $X$ is $F$-split;
\item[$(1$-$1)$] $b_1(X)=2$, $\kappa(X)=-\infty$ and 
$X$ is the projective bundle $\mathbb P(\mathcal E)$ associated with a rank two vector bundle $\mathcal E$ on $A$; 
\item[$(1$-$2)$] $b_1(X)=2$, $\kappa(X)=0$ and 
$X$ is a hyperelliptic surface such that every closed fiber of $a$ is an ordinary elliptic curve; 
\item[$(2)$\hspace{10pt}] $X$ is an abelian surface.
\end{itemize}
In the case of $(1$-$1)$, the morphism $a$ is $F$-split if and only if either  
\begin{itemize}
\item[$(\rm a)$] $\mathcal E$ is decomposable, 
\item[$(\rm b)$] $\mathcal E$ is indecomposable, $p>2$ and $\deg\mathcal E$ is odd, or
\item[$(\rm c)$] $\mathcal E$ is indecomposable, $p=2$ and $A$ is ordinary.
\end{itemize} 
In the case of $(0)$, $(1$-$2)$ and $(2)$, the morphism $a$ is $F$-split.
\end{thm}
Note that the first Betti number $b_1(X)$ is equal to $2\dim A$. 
By Theorem~\ref{thm:ch_ab_intro}, we see that $b_1(X)=0$, $2$ or $4$. 
If $b_1(X)=0$, then the $F$-splitting of $a$ is equivalent to the $F$-splitting of $X$. 
If $b_1(X)=4$, then $X$ is an abelian surface as shown in Theorem~\ref{thm:ch_ab_intro}. 
The case when $b_1(X)=2$ is dealt with in the remainder of this section. 
As shown by Lemma~\ref{lem:kappa}, we have $\kappa(X)\le0$. 
\subsection{The case $b_1(X)=2$ and $\kappa(X)=0$}\label{subsection:1_0}
In this case, according to Bombieri and Mumford's classification of minimal surfaces with Kodaira dimension zero \cite{BM2}, 
we find that $X$ is a hyperelliptic or quasi-hyperelliptic surface.
If $a$ is locally $F$-split, then Proposition~\ref{prop:fiber}
shows that $a$ has normal geometric generic fiber, so $X$ is hyperelliptic. 
Hence, $X$ can be obtained as the quotient of 
the product $E_0\times E_1$ of two elliptic curves $E_0$ and $E_1$ 
by an action of a finite group $B$ \cite[Theorem 4]{BM2}. 
Then, $A\cong E_1/B$ and every closed fiber of $a$ is isomorphic to $E_0$. 
\begin{prop}\label{prop:hypell}
The followings are equivalent:
\begin{itemize}
\item[$(1)$] $a$ is $F$-split;
\item[$(2)$] $a$ is locally $F$-split;
\item[$(3)$] $E_0$ is ordinary.
\end{itemize}
\end{prop}
\begin{proof}
The implication (1)$\Rightarrow$(2) is obvious. 
If $a$ is locally $F$-split, Proposition~\ref{prop:fiber} says that the general fiber is $F$-split, 
so $E_0$ is $F$-split, which shows (2)$\Rightarrow$(3). 
We prove (3)$\Rightarrow$(1). Assume that $E_0$ is $F$-split.
Since $a$ is flat and every fiber has trivial canonical bundle, 
$K_X\sim a^*C$ for a Cartier divisor $C$ on $A$.
We see from the assumption that $C\sim_{\mathbb{Q}}0$. 
Theorem~\ref{thm:cbf}~(ii)-(3) then tells us that $a$ is $F$-split, which is our claim. 
\end{proof}
\subsection{The case $b_1(X)=2$ and $\kappa(X)=-\infty$}\label{subsection:1_-inf}
In this case, $X$ is a ruled surface over an elliptic curve. 
We start with recalling some facts about vector bundles on elliptic curves. 
In the following theorem and lemmas, we let $C$ be an elliptic curve.
\begin{thm}\label{thm:facts_ell}
Let $\mathcal E_C(r,d)$ be the set of isomorphism classes of indecomposable vector bundles 
on $C$ of rank $r$ and of degree $d$.
\begin{itemize}
\item[(1)]\cite[Theorem 10]{Ati57} 
For every $\mathcal E$, $\mathcal E'\in\mathcal E_C(r,d)$, there exists an $\mathcal L\in\mathrm{Pic}^0(C)$ such that $\mathcal E\otimes\mathcal L\cong \mathcal E'$. 
Take $\mathcal L_1$, $\mathcal L_2\in\mathrm{Pic}^0(C)$. 
Then, $\mathcal E\otimes\mathcal L_1\cong\mathcal E\otimes\mathcal L_2$ if and only if $\mathcal L_1^{r'}\cong\mathcal L_2^{r'}$, where $r':=r/(r,d)$. 
Furthermore, when $d=0$, there exists a unique element $\mathcal E_{r,0}$ in $\mathcal E_C(r,0)$ such that $H^0(C,\mathcal E_{r,0})\ne0$.
\item[(2)]\cite[Proposition 2.1]{Oda71} 
Let $\pi:C'\to C$ be an isogeny of degree $r$ and $\mathcal L$ a line bundle of degree $d$ on $C'$.
If $r$ and $d$ are coprime, then $\pi_*\mathcal L \in \mathcal E_C(r,d)$.
\item[(3)]\cite[Theorem 2.16]{Oda71} 
Let $r>0$ and $d$ be coprime integers and suppose that $\mathcal E\in\mathcal E_C(rh,dh)$ for some $h>0$. 
When $C$ is ordinary, ${F_C}^*\mathcal E$ is indecomposable. 
When $C$ is supersingular, ${F_C}^*\mathcal E$ is indecomposable if and only if 
either $h\ne 1$ and $p\nmid r$, or $h=1$.
\end{itemize}
\end{thm}
\begin{lem}\label{lem:inv}
Let $\mathcal F$ be a vector bundle on $C$ of rank $r$ 
and $\mathcal L$ a line bundle such that $\mathcal F\otimes\mathcal L\cong\mathcal F$.
Then $\mathcal L^r\cong\O_C$.
\end{lem}
\begin{proof}
This follows from $(\det\mathcal F)\otimes\mathcal L^r\cong \det(\mathcal F\otimes\mathcal L)\cong\det\mathcal F$.
\end{proof}
\begin{lem}\label{lem:tors}
Let $\pi:C'\to C$ be a finite morphism of degree $d$ from an elliptic curve $C'$. 
Let $\mathcal L$ be a line bundle on $C$ such that $\pi^*\mathcal L\cong \O_{C'}$. 
Then $\mathcal L^d\cong\O_C$.
\end{lem}
\begin{proof}
The projection formula shows $(\pi_*\O_{C'})\otimes\mathcal L\cong\pi_*\O_{C'}$, 
so the claim follows from Lemma~\ref{lem:inv}.
\end{proof}
In characteristic zero, the pullback of $\mathcal E_{r,0}$ 
to another elliptic curve via a finite morphism is again indecomposable.
However, the lemma below shows that, in positive characteristic, 
the pullback of $\mathcal E_{r,0}$ can be trivial. 
\begin{lem}\label{lem:trivialize} 
There exists a finite morphism $\pi:C'\to C$ from an elliptic curve $C'$ such that the following conditions hold:
\begin{itemize}
\item[$(1)$] $\pi^*\mathcal E_{2,0}\cong \O_{C'}\oplus\O_{C'}$;
\item[$(2)$] $\pi_*\O_{C'}\cong \mathcal E_{p,0}$; 
\item[$(3)$] $\pi^*:\mathrm{Pic}^0(C)\to\mathrm{Pic}^0(C')$ is injective. 
\end{itemize}
\end{lem}
\begin{proof}
Since $\mathcal E_{2,0}$ is obtained as a nontrivial extension $\xi$ of $\O_C$ by $\O_C$, 
it is enough to find a finite morphism $\pi:C'\to C$ from 
an elliptic curve such that $\pi^*$ kills $H^1(C,\O_C)$. 
If $C$ is ordinary, i.e., ${F_C}^*:H^1(C,\O_C)\to H^1(C,\O_C)$ is an isomorphism,
then we may assume that ${F_C}^*\xi_1=\xi_1$.
In this case, $\xi_1$ defines an \'etale cover $\pi:C'\to C$ of degree $p$ such that $\pi^*\xi_1=0$.
If $C$ is supersingular, or equivalently ${F_C}^*\xi_1=0$, then we set $\pi:=F_C$.
Next, we prove (2). 
Let $\mathcal F$ be an indecomposable direct summand of $\pi_*\O_{C'}$. 
Theorem~\ref{thm:facts_ell} says that $\mathcal F\cong\mathcal E_{r,0}\otimes\mathcal L$ for an $\mathcal L\in\mathrm{Pic}^0(C)$, 
where $r:=\mathrm{rank}\,\mathcal F$. 
Then, by the projection formula, we have 
$$
\mathcal E_{r,0}\cong\mathcal F\otimes\mathcal L^{-1}\subseteq (\pi_*\O_{C'})\otimes\mathcal L^{-1}\cong\pi_*\pi^*\mathcal L^{-1}. 
$$
This means that $\pi^* \mathcal L^{-1}$ is a numerically trivial line bundle having non-zero global sections, 
so it is trivial, and hence Lemma~\ref{lem:tors} implies that $\mathcal L$ is $p$-torsion. 
If $C$ is ordinary, then $\pi$ is \'etale, so the commutative diagram 
$$
\xymatrix{C' \ar[r]^{F_{C'}} \ar[d]_{\pi} & C' \ar[d]^{\pi} & \\ 
C \ar[r]^{F_C} & C & }
$$
is cartesian, which implies $F_{C}^*\pi_*\O_{C'}\cong\pi_*\O_{C'}$. 
From this, we find $\pi_*\O_{C'}\cong\mathcal E_{p,0}$. 
If $C$ is supersingular, then $C$ has no nontrivial $p$-torsion line bundle, so $\mathcal L\cong\O_C$.
This means that $\pi_*\O_{C'}\cong\mathcal E_{p,0}$. 
Finally, we show (3). Take $\mathcal L\in\mathrm{Pic}(C)$ with $\pi^*\mathcal L\cong\O_{C'}$. 
Combining (2) with the projection formula, 
we get $\mathcal E_{p,0}\otimes\mathcal L\cong\mathcal E_{p,0}$, 
so Theorem~\ref{thm:facts_ell}~(1) says that $\mathcal L \cong \O_C$, 
which is our assertion. 
\end{proof}
\begin{lem}\label{lem:sym}
The $m$-th symmetric product $S^m\mathcal E_{2,0}$ of $\mathcal E_{2,0}$ is isomorphic to a direct sum of vector bundles of the form $\mathcal E_{r,0}$.
\end{lem}
\begin{proof}
Let $\mathcal F$ be an indecomposable direct summand of $S^m\mathcal E_{2,0}$. 
Theorem~\ref{thm:facts_ell} says that $\mathcal F\cong\mathcal E_{r,0}\otimes\mathcal L$ for an $\mathcal L\in\mathrm{Pic}^0(C)$, 
where $r:=\mathrm{rank}\,\mathcal F$. 
Let $\pi:C'\to C$ be as in Lemma~\ref{lem:trivialize}. 
Since $\pi^*S^m\mathcal E_{2,0}$ is trivial, we have $\pi^*\mathcal L\cong\O_{C'}$, 
so Lemma~\ref{lem:trivialize}~(3) shows that $\mathcal L \cong \O_C$, 
which completes the proof. 
\end{proof}
We now return to the study of the Albanese morphism $a:X\to A$ of $X$. 
We may regard $a:X\to A$ as the natural projection $\mathbb P(\mathcal E)\to A$ 
of the projective bundle $\mathbb P(\mathcal E)$ for a vector bundle $\mathcal E$ on $A$ of rank two. 
If $\mathcal E$ is decomposable, then $a$ is $F$-split as shown in Example~\ref{eg:pn-bdl}. 
Hence, we assume that $\mathcal E$ is indecomposable in the remaining of this section. 
We only need to consider the case when $\deg\mathcal E=0$ or $1$. 
\subsubsection{The case when $\mathcal E$ is an indecomposable vector bundle of degree 0.} \label{subsubsection:0}
In this case, we may assume that $\mathcal E=\mathcal E_{2,0}$ by Theorem~\ref{thm:facts_ell}~(1). 
Lemma~\ref{lem:trivialize} then says that we have a finite morphism $\pi:A'\to A$ 
from an elliptic curve $A'$ such that $\pi^*\mathcal E_{2,0}\cong\O_{A'}^{\oplus 2}$. 
In particular, $X_{A'}\cong \mathbb P(\pi^*\mathcal E_{2,0})\cong \mathbb P^1\times A^1$. 
We show the following:
\begin{prop}\label{prop:deg0}
The morphism $a:X\to A$ is $F$-split if and only if $A$ is ordinary and $p=2$.
\end{prop}
To prove Proposition~\ref{prop:deg0}, we prepare the claims below.
\begin{cl}\label{cl:Iitaka}
There exists an algebraic fiber space $g:X\to Y\cong\mathbb P^1$ such that $g^*\O_Y(1)\cong\O_X(p)$.
\end{cl}
\begin{cl}\label{cl:ds}
The $p$-th symmetric power $S^p\mathcal E_{2,0}$ of $\mathcal E_{2,0}$ is isomorphic to $\mathcal E_{p,0}\oplus\O_A.$ 
\end{cl}
\begin{proof}[Proof of Claims~\ref{cl:Iitaka} and~\ref{cl:ds}]
Since $\pi^*$ preserves the Iitaka-Kodaira dimension (cf.~\cite[Theorem 10.5]{Iit82}), 
we have 
$$
1\le \kappa(X,\O_{X}(1))
=\kappa(X_{A'},\O_{X_{A'}}(1))
=\kappa(\mathbb P^1,\O_{\mathbb P^1}(1))
=1, 
$$ 
so $\O_X(1)$ is semi-ample.
Let $g:X\to Y$ be the Iitaka fibration associated to $\O_X(1)$. 
Then, $Y\cong \mathbb P^1$ and a general fiber $B$ of $g$ is an elliptic curve. 
Now we have the following commutative diagram:
$$
\xymatrix{B_{A'}\ar[r] \ar[d] & B \ar[d] & \\ 
X_{A'} \ar[r]^{\pi_X} \ar[d]_{a_{A'}} & X \ar[d]^{a} \ar[r]^g & Y \\ 
A' \ar[r]_{\pi} & A. & }
$$
By the construction, we have $\O_X(B)\cong \O_X(m)\otimes a^*\mathcal L$ 
for an $m\ge 2$ and a torsion line bundle $\mathcal L$ on $A$.
Take $l\in\mathbb{Z}$. 
We consider the exact sequence 
\begin{align*}
0\to \O_X(l)\otimes\O_X(-B)\to \O_X(l)\to \O_B\to 0 
\tag{\ref{cl:ds}.1} \label{seq:surf_l}
\end{align*}
of $\O_X$-modules. 
Pushing forward via $a$, when $l=m-1$, we get $S^{m-1}\mathcal E_{2,0}\cong a_*\O_B$. %  \tag{\ref{cl:ds}.3} \label{seq:surf_m-1}
This shows $H^0(A,S^{m-1}\mathcal E_{2,0})\cong k$, 
so Lemma~\ref{lem:sym} says $S^{m-1}\mathcal E_{2,0}\cong\mathcal E_{m,0}$, 
which means $\pi^*a_*\O_B\cong\O_{A'}^{\oplus m}$. 
Since $\pi:A'\to A$ and $a|_B:B\to A$ are flat, we obtain 
\begin{align*}
p\ge\dim H^0(B,(a|_B)^*\mathcal E_{p,0})
&=\dim H^0(B,(a|_B)^*\pi_*\O_{A'})\\
&=\dim H^0(B_{A'},\O_{B_{A'}})
=\dim H^0(A',\pi^*a_*\O_B)=m\ge2. 
\end{align*}
If $p>m$, then $(a|_B):\O_A\to (a|_B)_*\O_B$ splits, 
so $(a|_B)^*\mathcal E_{p,0}\cong\mathcal E_{p,0}$, which contradicts the above. 
Hence, $p=m$. 
Pushing forward (\ref{seq:surf_l}) via $a$ again, when $l=p$, we obtain the exact sequence 
\begin{align*}
0\to \mathcal L^{-1}\to S^p\mathcal E_{2,0}\to a_*\O_B\to0.  \tag{\ref{cl:ds}.2} \label{seq:surf_p}
\end{align*}
Combining this with Lemma~\ref{lem:sym}, 
we see that $\mathcal L\cong\O_A$, i.e., $g^*\O_Y(1) \cong \O_X(B) \cong \O_X(p)$, 
which is Claim~\ref{cl:Iitaka}. 
Also, since $H^0(A,S^p\mathcal E_{2,0})\cong H^0(X,\O_X(p))\cong H^0(Y,\O_Y(1))=k^{\oplus 2}$, 
we see that (\ref{seq:surf_p}) splits, which is Claim~\ref{cl:ds}. 
\end{proof}
We now start the proof of Proposition~\ref{prop:deg0}.
\begin{proof}[Proof of Proposition~\ref{prop:deg0}]
We use the same notation as the proof of Claims~\ref{cl:Iitaka} and~\ref{cl:ds}.  
We first prove that $(a,B)$ is $F$-split, assuming that $A$ is ordinary and $p=2$. 
We now have $\o_X\otimes\O_X(B)\cong\O_X(-2)\otimes\O_X(p)\cong\O_X$. 
Hence, thanks to Theorem~\ref{thm:cbf}~(ii)-(3), 
it is enough to show that $(X_z,B|_{X_z})$ is $F$-split for a fiber $X_z\cong\mathbb P^1$ of $a$.
Since $\pi^*a_*\O_B\cong\O_{A'}\otimes\O_{A'}$, 
we see that $B_{A'}$ is a disjoint union of two sections of $a_{A'}:X_{A'}\to A'$. 
This implies that the divisor $B|_{X_z}$ is a sum of two distinct (reduced) points.
Using the assumption that $p=2$, we find that $(X_z,B|_{X_z})$ is $F$-split. 

Next, we show that $A$ is ordinary and $p=2$, assuming $a$ is $F$-split. 
Fix an $e>0$ such that $\phi^{(e)}_{X/A}$ splits. 
By Claim~\ref{cl:Iitaka}, we have 
\begin{align*}
H^0(X,\o_X^{1-p^e})
&=H^0(X,\O_X(2p^e-2)) \\ 
&=H^0\left(X,\O_X(p-2)\otimes g^*\O_Y(2p^{e-1}-1)\right)=V\cdot g^*W, 
\end{align*}
where $V:=H^0(X,\O_X(p-2))$ and $W:=H^0(Y,\O_Y(2p^{e-1}-1))$. 
Therefore, we can choose $v\in V$ and $w\in W$ such that $v\cdot g^*w\in H^0(X,\o_X^{1-p^e})$ 
gives a splitting of $\phi^{(e)}_{X/A}$. 
The pullback $({\pi_X}^*v)\cdot({\pi_X}^*g^*w)$ then gives a splitting of $\phi^{(e)}_{X_{A'}/A'}$. 
Suppose that $A$ is supersingular. Then, by definition, $\pi=F_A$. 
One can easily check that $g\circ\pi_X$ factors through $F_Y:Y\to Y$, 
so there is some $u\in H^0(\O_{X_{A'}}(2p^{e-1}-1))$ with ${\pi_X}^*g^*w=u^p$.
This implies that the pullback of $v\cdot u^p$ to a fiber of $a_{A'}$ induces 
a splitting of $\mathbb P^1$, which is a contradiction. 
Hence, $A$ must be ordinary. 
We show $p=2$. 
It follows from the $F$-splitting of $a$ that 
$\O_{X_{A^1}}\to {F_{X/A}^{(1)}}_*\O_{X^1}$ splits. 
Applying the functor ${a_{A^1}}_*(\underline{\quad}\otimes\O_{X_{A^1}}(1))$ to this, we get 
$$
\mathcal E_{2,0}\cong{F_A}^*\mathcal E_{2,0}\cong {a_{A^1}}_*\O_{X_{A^1}}(1)
\xrightarrow{\textup{{\tiny split}}} {a^{(1)}}_*\O_{X^1}(p)
\cong S^p\mathcal E_{2,0}
\overset{\textup{{\tiny Claim~\ref{cl:ds}}}}{\cong}\mathcal E_{p,0}\oplus\O_A, 
$$ 
so $\mathcal E_{2,0}\cong \mathcal E_{p,0}$, and thus $p=2$, which completes the proof. 
\end{proof}
\subsubsection{The case when $\mathcal E$ is an indecomposable vector bundle of degree 1.} 
\begin{prop}\label{prop:deg1}
Suppose that $\deg\mathcal E=1$. Then, $a$ is $F$-split if and only if $A$ is ordinary or $p>2$.
\end{prop}
\begin{proof}
We first prove that $a$ is $F$-split, assuming that $A$ is ordinary or $p>2$. 
When $p>2$, we take the \'etale cover $\rho:A'\to A$ of degree two 
corresponding to a torsion line bundle $\mathcal L$ of order two.
Then 
$$
{\rho_{A'}}_*\O_{A'\times_{A}A'}
\overset{\textup{{\tiny flatness of $\rho$}}}{\cong}\rho^*\rho_*\O_{A'}
\overset{\textup{{\tiny definition of $\rho$}}}{\cong}\rho^*(\O_A\oplus\mathcal L)
\overset{\textup{{\tiny definition of $\rho$}}}{\cong}\O_{A'}\oplus\O_{A'}.
$$ 
Hence $A'\times_A A'$ is a disjoint union of two copies of $A'$.
By Theorem \ref{thm:facts_ell} (1) and (2), we see that 
$\rho_*\mathcal M\cong\mathcal E$ for a line bundle $\mathcal M$ on $A'$ of degree 1. 
Then, 
$$
\rho^*\mathcal E\cong \rho^*\rho_*\mathcal M \cong {\rho_{A'}}_*\mathcal M_{A'} \cong \mathcal M\oplus\mathcal M.
$$
Therefore, $X':=X_{A'}\cong \mathbb P(\mathcal M\oplus\mathcal M)$ is $F$-split over $A'$. 
We now have the following commutative diagram: 
$$
\xymatrix@R=25pt@C=25pt{ {X'}^{e} \ar[r] \ar[d]_{F_{X'/A'}^{(e)}} & X^e \ar[d]^{F_{X/A}^{(e)}} \\ 
{X'}_{{A'}^e} \ar[r]_{(\rho^{(e)})_X} & X_{A^e}. }
$$
This induces the commutative diagram of $\O_{X_{A^e}}$-modules 
$$
\xymatrix@R=25pt@C=25pt{\O_{X_{A^e}}\ar[r] \ar[d] \ar[dr] & {(\rho^{(e)})_X}_*\O_{{X'}_{{A'}^e}} \ar[d] \\
{F_{X/A}^{(e)}}_*\O_{X^e} \ar[r] & {(\rho^{(e)})_X}_*{F_{X'/A'}^{(e)}}_*\O_{{X'}^e}. } 
$$
By the above argument, the right vertical morphism splits.
Since $p\nmid\deg\rho$, the upper horizontal morphism splits, 
so the diagonal one also splits, 
and hence so does the left vertical one, i.e., $a:X\to A$ is $F$-split. 
When $p=2$ and $A$ is ordinary, Theorem~\ref{thm:facts_ell}~(3) tells us that $F_A^*\mathcal E\in\mathcal E_A(2,2)$. 
Proposition~\ref{prop:deg0} then shows that $a_{A^1}:X_{A^1}\to A^1$ is $F$-split.
Replacing $\rho$ by $F_A$, we can prove the assertion by the same argument as the above. 

Next, we assume that $p=2$ and $A$ is supersingular. 
Theorem~\ref{thm:facts_ell}~(3) then shows ${F_A}^*\mathcal E\in\mathcal E_A(2,2)$. 
Hence, Proposition~\ref{prop:deg0} tells us that $a_{A^1}:X_{A^1}\to A^1$ is not $F$-split. 
This requires that $a:X\to A$ is not $F$-split, which completes the proof.
\end{proof}
\bibliographystyle{abbrv}
\bibliography{ref.bib}
\end{document}